\documentclass[10pt,reqno]{amsart}
\usepackage{amsmath, latexsym, amsfonts, amssymb,amsthm, amscd,epsfig,enumerate}
\usepackage{subfigure,color, float, caption}
\usepackage{tikz,colortbl, bm, hyperref, mathpazo}
\usetikzlibrary{calc}
\usetikzlibrary{decorations.pathreplacing}

\advance\hoffset -.75cm

\definecolor{refkey}{gray}{.75}
\definecolor{labelkey}{gray}{.75}

\newcommand{\R}{\mathbb R}

\newcommand{\N}{\mathbb N}

\newcommand{\diff}{\mathrm{d}}

\newcommand{\boldmu}{{\bm{\mu}}}
\newcommand{\boldnu}{{\bm{\nu}}}
\newcommand{\Mfrak}{{\mathfrak{M}}}
\newcommand{\Nfrak}{{\mathfrak{N}}}

\newcommand{\pr}{\mathbb P}

\newcommand{\ident}{{\mathchoice {\rm 1\mskip-4mu l} {\rm 1\mskip-4mu l}
{\rm 1\mskip-4.5mu l} {\rm 1\mskip-5mu l}}}

\newcommand{\gegerm}{{\,\ge_{\textrm{germ}\,}}}

\newcommand{\gepgf}{{\,\ge_{\textrm{pgf}\,}}}

\newtheorem{Theorem}{Theorem}[section]

\newtheorem{Corollary}[Theorem]{Corollary}
\newtheorem{Remark}[Theorem]{Remark}
\newtheorem{Proposition}[Theorem]{Proposition}
\newtheorem{Definition}[Theorem]{Definition}
\newtheorem{Example}[Theorem]{Example}

\newtheorem{Assumption}[Theorem]{Assumption}

\title%[Recent results on branching random walks]
{Critical parameters of germ-monotone families of branching random walks}
\author[D.~Bertacchi]{Daniela Bertacchi}
\address{D.~Bertacchi, Dipartimento di Matematica e Applicazioni,
Universit\`a di Milano--Bicocca,
via Cozzi 53, 20125 Milano, Italy.}
\email{daniela.bertacchi\@@unimib.it}

\author[F.~Zucca]{Fabio Zucca}
\address{F.~Zucca, Dipartimento di Matematica,
Politecnico di Milano,
Piazza Leonardo da Vinci 32, 20133 Milano, Italy.}
\email{fabio.zucca\@@polimi.it}

\date{}

\begin{document}
%\nocite{*}

\begin{abstract}
We introduce a broad class of families of branching random walks on a countable set $X$, which we refer to as germ-monotone branching random walks (GMBRWs). The processes in each family are parametrized by a positive parameter $\lambda>0$, which controls the overall reproductive speed, and they are monotonically increasing in $\lambda$ with respect to the germ order, a notion that extends classical stochastic domination. This framework encompasses a wide range of models, including classical continuous-time branching random walks, as well as discrete-time counterparts of certain non-Markovian processes such as ageing branching random walks.
\hfill\break\noindent
We define a general notion of critical parameter $\lambda(A)$ associated with each subset $A \subseteq X$, which serves as a threshold separating almost sure extinction in $A$ from positive probability of survival in $A$. This unifies and extends the classical global and local critical parameters $\lambda_w$ and $\lambda_s$, which can be recovered as special cases.
We then investigate how modifications of the reproduction laws, either on a finite set or on a more general subset of $X$, affect these critical parameters. Our results extend earlier contributions in the literature.
\end{abstract}

\maketitle
\noindent {\bf Keywords}: branching random walk, branching process, germ order, critical parameters, local survival, global survival, pure global survival phase.

\noindent {\bf AMS subject classification}: 60J05, 60J80.

\section{Introduction}

A branching process, also known as the Galton–Watson process (see~\cite{cf:GW1875}),
is a stochastic process in which a particle dies and produces a random 
number of offspring according to a prescribed offspring distribution 
$\rho$, where $\rho(n)$ denotes the probability that exactly $n$ 
children are born. Different particles reproduce independently following
the same law. The process may either become extinct (that is, no 
particles remain alive after some finite time) almost surely, or survive
indefinitely with positive probability. The extinction probability can 
be explicitly characterized in terms of the offspring distribution 
$\rho$.

A branching random walk
(BRW, hereafter) generalizes this model by associating each particle 
with a location $x \in X$, where $X$ is an at most countable set. Although 
$X$ is often interpreted as a spatial domain, it can also represent a 
set of types (see, for example,~\cite{cf:KurtzLyons}). Particles located
at a site $x \in X$ are replaced by a random number of offspring, which
are then distributed among the sites in $X$. The reproduction law depends 
on the location of the parent particle, and all particles reproduce 
independently. 
Branching processes can be seen as the special case of BRWs
where the space set $X$ is a singleton.

This class of processes has been extensively studied; see, for instance,~\cite{cf:AthNey, cf:Big1977, cf:Harris63}
for early contributions. Beyond the classical setting of 
$\mathbb{Z}^d$, branching random walks have been investigated on finite 
sets~\cite{cf:MountSchi}, on trees~\cite{cf:Attia, cf:Ligg1, cf:MadrasSchi, cf:Muller1, cf:PemStac1, cf:Su},
on general graphs~\cite{cf:Cand1, cf:Cand2, cf:Duss, cf:KW, cf:KSava}  and in random environment~\cite{cf:MP00, cf:MP03}.
%  In the case of a BRW, there is no upper bound on the number of 
% particles that may occupy a single site. Imposing such an upper bound, say at 
% most $m$ particles per site, leads to an $m$-type contact process. The 
% branching random walk can be recovered as the limit of these processes 
% as $m \to \infty$~\cite{cf:PemStac1, cf:Z1}.

Branching random walks may be 
defined either in discrete time, where each particle lives for exactly 
one generation and all offspring appear in the next generation, or in 
continuous time, where generations overlap because particles 
reproduce at different times during their lifetimes. In continuous time,
one often considers families of parametrized processes in which the 
reproduction rates between locations are fixed, and a parameter $\lambda> 0$ 
controls the overall reproductive speed (larger values of 
$\lambda$ correspond to shorter intervals between successive 
reproduction events;
see Section~\ref{sec:basic} for details).
%FRASE SUCCESSIVA ANTICIPATA\\
These families are monotonically increasing with respect to the stochastic domination, that is, if $\lambda>\lambda^*$, then the $\lambda$-process stochastically dominates the $\lambda^*$-process.

This monotonicity  implies that processes with larger values of $\lambda$ are more likely to survive, and naturally leads to the definition
of \textit{critical parameters}. For a branching process, there is only one critical parameter $\lambda_c$, which is the threshold between
almost sure extinction and survival with positive probability.
%FRASE SUCCESSIVA ANTICIPATA\\
Since a BRW evolves on an 
underlying spatial structure, its behavior is typically more complex.
In particular, one may  
distinguish between global survival (on the entire space $X$) and  local
survival (in a given site).
 The threshold separating global 
extinction from global survival is known as the global critical parameter $\lambda_w$, while the threshold separating local extinction from local survival is called the local critical parameter
$\lambda_s$. The subscripts $w$ and $s$ stand for “weak” and “strong”, 
which are often used as synonyms for global and local survival, 
respectively. However, throughout this paper, the term "strong" is reserved for a distinct notion of survival (see Definition~\ref{def:survival}). Several authors have addressed the identification of these
critical parameters~\cite{cf:PemStac1, cf:Stacey} and established criteria for survival and extinction~\cite{cf:Gantert1, cf:Gantert2, cf:MachadoMenshikovPopov, cf:Muller1, cf:Muller2, cf:Su}.

In this paper, we generalize the concept of critical parameter in two directions. 
First, we observe that, given a family  of processes, parametrized by $\lambda$, in order to define values of $\lambda$ that
act as a threshold between extinction and survival, we do not need stochastic monotonicity in $\lambda$.
What is needed is that, if $\lambda<\lambda^*$,
 whenever there is a positive probability of survival for a $\lambda$-process,
then the same holds for a $\lambda^*$-process. This monotonicity is granted
if the $\lambda$-process is smaller than the $\lambda^*$-process,  according to the \textit{germ-order}.
This order is based on the behavior of the generating function of the process 
(see Definition~\ref{def:ordering}). It has been introduced in \cite{cf:Hut2022} and further extended in \cite{cf:BZgerm}.
We are thus able to define the critical parameters for parametrized families which we name
\emph{germ-monotone BRWs} (briefly \emph{GMBRWs},  see Definition~\ref{def:onepardependentfamily}). 
These families contain, among others, the \emph{discrete-time counterpart} of continuous time BRWs and of BRWs with ageing (see Section~\ref{subsec:continuous}). 
Secondly, we extend the notion of critical parameter to any given subset $A\subseteq X$.
We define the value  $\lambda(A)$  as the threshold separating almost sure extinction in $A$ and positive probability of survival in $A$ (see Definition~\ref{def:survival}). The classical critical parameters are particular cases of this definition:
 $\lambda_w=\lambda(X)$ and $\lambda_s=\lambda(A)$ where $A$ is a finite, nonempty set.

Since critical parameters are the thresholds between extinction and survival, it is not only relevant to 
identify their values, but also to understand how these parameters may change when the structure of the process
is somehow modified.
Indeed, in ecological or epidemiological 
applications, we may view the parameter $\lambda$ as a reproduction/infection speed which is an intrinsic characteristic
of the population that we are modelling. If we aim at pest or disease control, we can either try and lower this speed
(which implies a global change of the process), or perform modifications of the process in some regions (so that
the critical parameters increase their values).
For example, in order to  eradicate an epidemic,  one possible strategy is to reduce 
$\lambda$ by wearing face masks or locking down some social activities. 
An alternative strategy is to modify the reproduction laws 
%$\boldmu_\lambda$ 
 on a subset
 $\Delta$.  Such a modification may increase the critical parameters and consequently change the behavior of the process for a fixed value of $\lambda$ ($\lambda$ might be supercritical before the changes and subcritical after).

The effect of \emph{local modifications} (namely, modifications confined to a finite set) on $\lambda_s$ and $\lambda_w$ for continuous-time BRWs has been studied in \cite{cf:BZ2025}.
In this paper, we consider the 
extended critical parameters of GMBRWs and study their behavior 
when the reproduction laws are modified on a generic, not necessarily finite, subset $\Delta$ (see Section~\ref{sec:survivalprob}).

To summarize, the present paper generalizes the framework of \cite{cf:BZ2025} in three distinct directions.\begin{enumerate} 
	\item We work within the broader class of GMBRWs, whereas \cite{cf:BZ2025} was restricted to discrete-time counterparts of continuous-time BRWs. The ageing BRWs introduced in Section~\ref{subsec:continuous} provide a natural motivation for this extension.
	\item Rather than focusing solely on the classical critical parameters $\lambda_w$ and $\lambda_s$, we develop a theory for the general critical parameters $\lambda(A)$. The existence of a nontrivial $\lambda(A)$ distinct from both $\lambda_w$ and $\lambda_s$, is established in Section~\ref{sec:examples}.
	\item We consider a strictly wider class of modifications, of which the finite modifications studied in \cite{cf:BZ2025} are a special case. 
	\end{enumerate}

The paper is organized as follows. Section~\ref{sec:basic} presents the basic definitions (discrete-time BRWs, extinction probabiities, classical continuous-time BRWs and their critical parameters). We recall 
Theorem~\ref{cor:equivalence},
which links survival/extinction in different sets with the respective extinction probabilities.
Section~\ref{sec:GMBRW} recalls the definition of germ order for discrete-time BRWs and uses it to define the class of GMBRWs and their critical parameters, associated to subsets of the underlying graph.
A characterization of the local critical parameter and a lower bound for general critical parameters is given in 
Proposition \ref{pro:criticalA} and is applied to the case of continuous-time BRWs in Corollary \ref{cor:criticalA}.

Section~\ref{sec:survivalprob} is devoted to the discussion of what happens to the critical parameters, when we modify
the process. The results rely on Theorem~\ref{th:modifiedBRW} and Theorem~\ref{th:modifiedBRW-CT}.
Earlier versions of these Theorems appeared in \cite[Theorem 4.2]{cf:BZ17} and \cite[Theorem 2.4]{cf:BZ2020}, and are further generalized in the present work. Proposition~\ref{pro:pureweak-nonstrong} analyzes the behavior of $\min(\lambda(A), \lambda(B))$ before and after a modification of the process. In particular, they show that 
if  $\lambda(A) < \lambda(B)$, then $\lambda(A)$ cannot increase as a consequence of the modification. 
Theorem \ref{th:mainmod} and Corollary \ref{cor:new1} show that if before the modification $\lambda(A) < \lambda(B)$, then only some mutual positions of the critical parameters are possible (see Figure \ref{fig:possible-lambda}).
Subsection \ref{subsec:max} is dedicated to the application of these results to the particular case of the local and global critical parameters and to the question whether it is possible to create a pure global survival phase in a GMBRW, by applying finite modifications to the process.

In Section~\ref{sec:examples} we show  that  the critical value $\lambda(A)$ can indeed differ from the classic critical values $\lambda_w$ and $\lambda_s$ (Example \ref{ex:T5T6}) and that there are families of BRWs which are ordered with respect to germ order, hence they are GMBRWs, but they are not pgf-ordered nor stochastically ordered
 (Example \ref{ex:nonstochmon}). Example \ref{ex:T5T6loop} shows how a local modification of Example \ref{ex:T5T6}
 (i.e. the addition of a local reproduction in one site) affects the values of the critical parameters, depending on the
 strength of this local reproduction.

%\textcolor{red}{OUTLINE SISTEMARE ALLA FINE}
%
%The last result of the section, Proposition~\ref{pro:max}, deals with $\max(\lambda(A), \lambda(B))$. These results are extremely powerful in the case of local modifications.
%%Proposition~\ref{pro:pureweak-nonstrong} and Corollary~\ref{cor:pureweak-nonstrong1.6} analyze the behavior of $\min(\lambda(A), \lambda(B))$, before and after a modification. In particular, they establish that either $\lambda(A) \ge \lambda(B)$, or else $\lambda(A)$ cannot increase as a result of the modification. The final result of the section, Proposition~\ref{pro:max}, concerns the behavior of $\max(\lambda(A), \lambda(B))$. These results are especially powerful in the context of local modifications. Section~\ref{sec:max} is devoted to this important special case. 
%One of the main consequences is that either $\lambda(A)=\lambda_s$ or $\lambda(A)$ cannot increase as a consequence of local modifications; this is a rather surprising result since the existence of a strictly larger critical value, $\lambda(B) > \lambda(A)$, implies that no local modification can increase the critical value related to $A$.  Section~\ref{subsec:max} illustrates how the results of \cite{cf:BZ2025} can be recovered and further strengthened. Finally, an example is presented in Section~\ref{sec:examples} which shows that the critical value $\lambda(A)$ can indeed differ from the classic critical values $\lambda_w$ and $\lambda_s$; this example shows how $\lambda_s$, $\lambda(A)$ and $\lambda_w$ vary under local modifications.

%%%%%%%%%%%%%%%%%%%%%%%%%%%%%%%%%%%%%%%%%%

\section{Basic definitions and preliminaries}
\label{sec:basic}

\subsection{Discrete-time Branching Random Walks.}
\label{subsec:q1}

Given an at most countable set $X$, we define a discrete-time BRW
as a process $\{\eta_n\}_{n \in \N}$,
where $\eta_n(x)$ is the number of particles alive at $x \in X$ at time $n$. 
%The dynamics is described as follows: 
Let $S_X:=\{f:X \to \N\colon \sum_yf(y)<\infty\}$ and let 
$\boldmu=\{\mu_x\}_{x \in X}$ be a family of probability measures
on the (countable) measurable space $(S_X,2^{S_X})$. 
A particle of generation $n$ at site $x\in X$ lives one unit of time;
after that, a function $f \in S_X$ is chosen at random according to the law $\mu_x$.
%This function describes the number of children and their positions, that is,
The original particle is replaced by $f(y)$ particles at
$y$, for all $y \in X$. The choice of $f$ is independent for all breeding particles.
The BRW is denoted by $(X,\boldmu)$.
If $X$ is a singleton, $(X,\boldmu)$ is the classical branching process, with reproduction law $\boldmu$.

In order to keep track of the sites where a particle living at $x$ may place its offspring, with positive
probability, we
define the \emph{first moment matrix}
$M=(m_{xy})_{x,y \in X}$ of the process, by
$m_{xy}:=\sum_{f\in S_X} f(y)\mu_x(f)$.
To  any
discrete-time BRW we associate a graph $(X,E_\boldmu)$, where $(x,y) \in E_\boldmu$  
if and only if $m_{xy}>0$.
We say that there is a path of length $n \in \mathbb{N}$ from $x$ to $y$, %and we write $x \stackrel{n}{\to} y$, 
if it is
possible to find a sequence $\{x_i\}_{i=0}^n$ 
such that $x_0=x$, $x_n=y$ and $(x_i,x_{i+1}) \in E_\boldmu$
for all $i=0, \ldots, n-1$ (observe that there is always a path of length $0$ from $x$ to itself). If there is a path 
from $x$ to $y$,
then we write $x \to y$; whenever $x \to y$ and $y \to x$ we write $x \rightleftharpoons y$.
The equivalence relation $\rightleftharpoons$ induces a partition of $X$: the
class $[x]$ of $x$ is called \emph{irreducible class of $x$}.
If the graph $(X,E_\boldmu)$ is \emph{connected} (that is, there is only one irreducible class),
then we say that the BRW is \emph{irreducible}.
Irreducibility implies that the progeny of any particle can spread to any site of the graph with positive probability. 

In order to avoid trivial situations where particles have one offspring almost surely, we assume
henceforth the following.
For all $x \in X$ there is a vertex $y \rightleftharpoons x$ such that
$\mu_y(f\colon  \sum_{w\colon w \rightleftharpoons y} f(w)=1)<1$,
that is, in every equivalence class (with respect to $\rightleftharpoons$)
there is at least one vertex where a particle has a positive probability of producing 
a number of children different from 1.

  \begin{Definition}\label{def:surv-event}
  The event $\mathcal{S}(A):=\{\limsup_{n \to \infty} \sum_{y \in A} \eta_n(y)>0\}$ is called \emph{survival in $A$}, while $\mathcal{E}(A):=\mathcal{S}(A)^\complement$ is called \emph{extinction in $A$}.
  \end{Definition}
  
  We consider initial configurations with only one particle placed at a fixed site $x$ and we
denote by $\pr^{x}$ 
the law of the corresponding process. Evolution of processes with more than one initial particle
is obtained by superimposition.

\begin{Definition}\label{def:extprob}
	Given a BRW $(X,\boldmu)$, $x\in X$ and $A\subseteq X$, the probability of extinction in $A$, starting with one particle at $x$, is defined as
	\[
	{\mathbf{q}}(x,A)
	:=\pr^{x}(\mathcal{E}(A)).
	\]
	We denote by ${\mathbf{q}}(A)$ the extinction probability vector, whose $x$-entry is ${\mathbf{q}}(x,A)$.
\end{Definition}
If $A=\{y\}$, we write ${\mathbf{q}}(x,y)$ instead of ${\mathbf{q}}(x, \{y\})$.
Note that
${\mathbf{q}}(x,A)$ depends on $\boldmu$. When we need to stress this dependence, we write 
$ {\mathbf{q}}^\boldmu(x,A)$.
Extinction probabilities have been the object of intense study during the last decades.
We refer the reader, for instance, to \cite{cf:BBHZ, cf:BZ4, cf:Haupt, cf:Z1}.

  Based on the extinction probabilities, we classify the behavior of the BRW, see Table \ref{tb:table1bis} for a summary.
\begin{Definition}\label{def:survival} 
	%$\ $
	Let $(X,\boldmu)$ be a BRW starting from one particle at $x\in X$ and let $A\subseteq X$.
	We say that
	\begin{enumerate}
		\item 
		the process \textsl{survives} in $A$,
		 if 		$
		{\mathbf{q}}(x,A)<1$
		 and it  \textsl{goes extinct}  in $A$, if $
		{\mathbf{q}}(x,A)=1$;
		\item
		the process \textsl{survives globally}, if ${\mathbf{q}}(x,X) <1$  \textcolor{blue}and it  \textsl{goes globally extinct}, if $
		{\mathbf{q}}(x,A)=1$;
		\item
		there is \textsl{strong survival} in $A$,	if
		$ 
		{\mathbf{q}}(x,A)={\mathbf{q}}(x,X)<1$
		%\bar {\mathbf{q}}(x)<1
		and \textsl{nonstrong survival} in $A$ if ${\mathbf{q}}(x,X)<{\mathbf{q}}(x,A)<1$;
		\item
		the process is in a \textsl{pure global survival phase} if
		%starting from $x$ if
		$
		%\bar {\mathbf{q}}(x)
		{\mathbf{q}}(x,X) <{\mathbf{q}}(x,x)=1
		$.
	\end{enumerate}
\end{Definition}

 We note that, while survival and extinction are tail events, with a slight abuse of notation, we say that the process survives,
 or that there is survival, when the survival event has positive probability (\emph{with positive probability} is omitted). 
 Similarly, we say that the process goes extinct, or that there is extinction, when the extinction event is almost sure
 (\emph{almost surely} is tacitly understood).
 Moreover, in contrast to global survival, we say that there is \emph{local} survival when ${\mathbf{q}}(x,x)<1$ (i.e. the process returns to the starting vertex infinitely many times, with positive probability).
 By a Borel-Cantelli argument, it is easy to prove that, in irreducible BRWs, ${\mathbf{q}}(x,x)={\mathbf{q}}(x,A)$ for all finite nonempty subsets $A$. Therefore, we refer to \emph{local survival} (without mentioning the set), whenever ${\mathbf{q}}(x,A)<1$ for every finite nonempty set $A$.
 \\
 We note that in general, Definition~\ref{def:survival} depends on the starting vertex. When the
process is irreducible, however, for any fixed $A \subseteq X$, survival/extinction in $A$
starting from $x$ are respectively equivalent to survival/extinction in $A$ starting from $y$, for all
$x, y \in X$. 
Equivalently, for a given $A \subseteq X$, the condition
$\mathbf{q}(x,A) < 1$ holds for some $x \in X$ if and only if it holds for all
$x \in X$.
Strong survival in $A$, by contrast, may still depend on the starting
vertex even in the irreducible case. If every site has a positive probability of
producing no offspring, then strong survival becomes independent of the starting
vertex as well (see \cite[Section~3]{cf:BZ14-SLS}).

Moreover, if a BRW is irreducible, $A,B \subset X$ are nonempty, finite subsets and $C \subset X$ is nonempty, then $\mathbf{q}(x,A)=\mathbf{q}(x,B) \ge \mathbf{q}(x,C)$ for all $x\in X$. The first equality holds since local survival does not depend on the target vertex and survival in a finite (nonempty) set is equivalent to local survival to a vertex in the set. The second inequality follows from the fact that every nonempty set contains a finite nonempty set.

 \begin{table}[h]
	\begin{tabular}{ |c|c| }
		\hline 
		\rowcolor{lightgray}
		Notation & Meaning  \\ \hline
		%&   \\  
		${\mathbf{q}}(x,A)$ &  probability of extinction in $A$, starting from $x$\\ 
		%&    \\ 
		\hline
		${\mathbf{q}}(A)$ &  vector whose components are ${\mathbf{q}}(\cdot,A)$ \\ 
		%&    \\ 
		\hline
		${\mathbf{q}}(x,A)=1$ &  (a.s.) extinction in $A$, starting from $x$  \\
		\hline
		${\mathbf{q}}(x,A)<1$ &  survival (w.p.p.) in $A$, starting from $x$ \\ 
		%&   \\ 
		\hline
		${\mathbf{q}}(x,A)={\mathbf{q}}(x,X)<1$ & strong survival in $A$, starting from $x$   \\ 
		%&    \\ 
		\hline
	\end{tabular}
	\captionof{table}{Extinction probabilities, surivival and extinction of discrete-time BRWs. }\label{tb:table1bis}
\end{table}
 
 It is well-known that extinction probability vectors are fixed points of the generating function $G_\boldmu:[0,1]^X \to [0,1]^X$ defined as follows:
\begin{equation}\label{eq:genfun}
	G_\boldmu({\mathbf{z}}|x):= \sum_{f \in S_X} \mu_y(f) \prod_{y \in X} {\mathbf{z}}(y)^{f(y)}, %=\E^x \big[\prod_{y \in X}\mathbf{z}(y)^{\eta_1(y)}\big],
\end{equation}
where $G_\boldmu({\mathbf{z}}|x)$ is the $x$-coordinate of $G_\boldmu({\mathbf{z}})$.
For the properties of this generating function, see for instance \cite[Section 2.1]{cf:BZgerm} and the references therein.

\subsection{Equivalent sets and comparison of extinction probability vectors}
\label{subsec:equivalentsets}

Given a BRW $(X,\boldmu)$,
we examine survival and define relations between subsets $A,B \subseteq X$.
\begin{Definition}\label{def:equivalentsets} 
Let $A,B \subseteq X$.
\begin{enumerate}
\item If  $\mathbb P^x (\mathcal{S}(B) \cap \mathcal S(A))=\pr^x(\mathcal{S}(A))$, for all $x \in X$, we write
$A \Rightarrow B$ and we say that survival in $A$ implies survival in $B$.
 \item  If $\mathbb P^x (\mathcal{S}(B) \cap \mathcal S(A))<\pr^x(\mathcal{S}(A))$, 
	for some $x \in X$,   we write  $A \nRightarrow B$ and we say that survival in $A$ does not imply survival in $B$.
 \item
If $A \Rightarrow B$ and $B \Rightarrow A$, we say that the two sets are equivalent and use the notation 
  $A \Leftrightarrow B$.
 \item If $A \nRightarrow B$ and $B \nRightarrow A$, we write $A \nLeftrightarrow B$.
 \end{enumerate}
 \end{Definition}
By Definition \ref{def:equivalentsets}, $A \Leftrightarrow A$ for all $A \subseteq X$. When it is necessary to stress the dependence on $\boldmu$,
 we write $\stackrel{\boldmu}{\Rightarrow}$ and $\stackrel{\boldmu}{\Leftrightarrow}$.
 Note that $A \Rightarrow B$ means that almost all trajectories surviving in $A$, survive also in $B$.
 Similarly, 
 $A \nRightarrow B$ means that, for some starting point, there is a positive probability of surviving in $A$ while going extinct in $B$, or equivalently
 $\mathbb P^x (\mathcal{E}(B) \cap \mathcal S(A))>0$.

Henceforth, we employ the natural componentwise partial order on vectors: namely,
$\mathbf{q}(A) \le \mathbf{q}(B)$ if and only if $\mathbf{q}(x,A) \le \mathbf{q}(x,B)$
for all $x \in X$. Accordingly, $\mathbf{q}(A) < \mathbf{q}(B)$ if and only if
$\mathbf{q}(x,A) \le \mathbf{q}(x,B)$ for all $x \in X$ and
$\mathbf{q}(x_0,A) < \mathbf{q}(x_0,B)$ for some $x_0 \in X$; in particular, we note that
$\mathbf{q}(A) \nleq \mathbf{q}(B)$ if and only if $\mathbf{q}(x_0,A) > \mathbf{q}(x_0,B)$
for some $x_0 \in X$.
We show that the relations between $A,B\subseteq X$ can be investigated through the order of the respective extinction probabilities.

The following result links the relations between different sets and their extinction probability vectors.
Its statement is a rephrasing of \cite[Corollary 5.1]{cf:BBHZ} and part of \cite[Theorem 4.1]{cf:BBHZ}, therefore we omit the proof.

\begin{Theorem}\label{cor:equivalence}
	Let $A,B \subseteq X$.
	\begin{enumerate}
		\item $A \Rightarrow B$ if and
		only if $\mathbf{q}(A) \ge \mathbf{q}(B)$.
		\item $A \nRightarrow B$ if and
		only if $\mathbf{q}(x,A) < \mathbf{q}(x,B)$, for some $x\in X$.
		\item $A \nRightarrow B$ if and
		only if there is a positive probability of survival in $A$ without ever visiting $B$, starting from some $x\in X \setminus B$.
		\item $A \nRightarrow B$ if and
		only if there is a positive probability of survival in $A$ and extinction in $B$, starting from some $x\in X$.
		\item  $A \Leftrightarrow B$ if and
		only if $\mathbf{q}(A)  = \mathbf{q}(B)$.
		\item $A \nLeftrightarrow B$ if and only if there is no order relation between $\mathbf{q}(A)$ and $\mathbf{q}(B)$.
	\end{enumerate}
\end{Theorem}

\begin{Remark} \label{rem:1}
	If $\Delta \Rightarrow A$ then $A \Leftrightarrow A \cup \Delta$. Equivalently, $\mathbf{q}(\Delta) \ge \mathbf{q}(A)$ implies $\mathbf{q}(A)=\mathbf{q}(A \cup \Delta)$.
	Indeed, survival in $A$ implies survival in $A \cup \Delta$; conversely, survival in $A \cup \Delta$ either implies survival in $A$ or survival in $\Delta$ which, in turn, implies survival in $A$.
	
	Moreover, if $(X,\boldmu)$ is irreducible, then $A \Rightarrow B$ for every nonempty $A,B \subseteq X$ where $A$ is finite 
	(by the same Borel-Cantelli argument that we discussed in Section~\ref{subsec:q1} after 
	Definition \ref{def:survival}).
	Moreover,  $\mathbf{q}^\boldmu(C)$ does	not depend on the choice of the finite nonempty set $C \subseteq X$.
	 Therefore, 
	$A \Leftrightarrow B$ for all $A,B \subseteq X$ nonempty and finite. 
\end{Remark}

\subsection{Classical continuous-time BRWs and critical parameters}
\label{subsec:classic-continuous}

Given an at most countable $X$ and a nonnegative matrix $K=(k_{xy})_{x,y\in X}$ such that $\sum_{y \in X} k_{xy}<+\infty$ for every $x \in X$, 
%		$K=(k_{xy})_{(x,y)\in X\times X}$,
one can define a family of continuous-time Branching Random Walks
$\{\eta_t\}_{t\ge0}$, where $\eta_t(x)$ represents the number of particles alive at time $t$ at site $x$, for any $x\in X$.
The family is indexed by the reproductive speed parameter $\lambda>0$ and 
we denote it by  $(X,K)$.
With a slight abuse of terminology it is often referred to
as \emph{the} continuous-time BRW  or as the 
 \emph{classical} continuous-time BRW associated to $K$.
Its dynamics is defined as follows: each particle has an exponentially distributed lifetime with parameter 1.
During its lifetime each particle alive at $x$ breeds into $y$ according to the arrival times of its own Poisson point process with
intensity $\lambda k_{xy}$ (representing the reproduction rate).

It is clear that given two values $\lambda_1<\lambda_2$, one can construct the associated processes $\{\eta^1_t\}_{t\ge0}$ and
$\{\eta^2_t\}_{t\ge0}$ such that $\eta^1_t(x)\le\eta^2_t(x)$ for all $x\in X$, $t\ge0$, almost surely. 
This is a consequence of the fact that the law of the $\lambda_2$-process stochastically dominates 
the law of the $\lambda_1$-process (see Definition \ref{def:ordering}).
The extinction probability vectors 
depend on $\lambda$ and are nonincreasing in $\lambda$. We use the notation
$\mathbf{q}(x,A|\lambda)$ to denote the probability of extinction at the set $A$, for the $\lambda$-BRW.
We remark that
in general.
% even when the BRW $(X,K)$ is homogeneous, the map
% 	HOMOGENEOUS MEANS???
	$\lambda \mapsto \mathbf{q}(A|\lambda)$ does not need to be continuous in $[0,1]^X$
	with respect to the pointwise convergence topology. 
	Explicit discontinuous examples 
	%when $C=X$ or $C$ is finite 
	can be found in \cite[Remark 4.8]{cf:Z1}.	 
%\begin{Remark}
%	\label{rem:finitesets}
%	Recall that if a BRW $(X, \boldmu)$ is irreducible, then $\mathbf{q}^\boldmu(C)$ does
%	not depend on the choice of the finite nonempty set $C \subseteq X$.
%	Furthermore, for a continuous-time BRW $(X,K)$, the map $\lambda \mapsto \mathbf{q}(A|\lambda)$
%	is nonincreasing, since $\boldmu_\lambda \succeq \boldmu_{\lambda^*}$ for all
%	$\lambda \ge \lambda^* > 0$ (see the discussion preceding equation~\eqref{eq:criticalparametersA}).
%	\\
%	We remark that, even when the BRW $(X,K)$ is homogeneous, the map
%	$\lambda \mapsto \mathbf{q}(A|\lambda)$ does not need to be continuous in $[0,1]^X$
%	with respect to the pointwise convergence topology. 
%	Explicit examples when $C=X$ or $C$ is finite can be found in
%	\cite[Remark 4.8]{cf:Z1}.	 
%\end{Remark}
The monotonicity in $\lambda$ allows to define the critical parameters of the family, for any fixed set $A\subseteq X$ 
(with the convention that $\inf(\emptyset)=0$):
\begin{equation}\label{eq:criticalcts}
		\lambda(x,A):=
		\inf \{\lambda>0\colon \,
		\mathbf{q}(x,A|\lambda)<1 \}.
\end{equation}
This definition generalizes the two critical parameters that can be found in the literature:
the \emph{global critical parameter} $\lambda_w(x):=\lambda(x,X)$ and the  \emph{local critical parameter}
 $\lambda_s(x):=\lambda(x,\{x\})$, see also Table \ref{tb:tablelambda}.
  Characterizations of $\lambda_w(x)$ and $\lambda_s(x)$  can be found in the literature; see, for example
  \cite{cf:BZ2, cf:BZ17, cf:BZ2025, cf:Ligg1, cf:PemStac1, cf:Z1}.

The main difference between discrete-time BRWs and continuous-time BRWs is that in the second case, generations overlap and individuals
of the $n$-th generation may be born after individuals of the $m$-th, even if $n<m$.
Nevertheless, if we are interested in the probability of extinction in any given set $A$, then a continuous-time BRW shares its
extinction probabilities with its \emph{discrete-time counterpart}.
This process is defined as follows: given a particle at generation 0 (i.e. alive at time 0), its children, born during its whole lifetime in the continuous-time process,  belong to the first generation
(i.e. are alive at time $n=1$). We proceed by induction: each particle alive at time $n$ is replaced at time $n+1$ by the children that its
corresponding particle in the continuous-time process generates during its lifetime.
Clearly, from the knowledge of the discrete-time counterpart of $(X,K)$, we cannot retrieve how many particles are alive at time $t$ (and where), but the information about survival and extinction is intact.
It is easy to show that the expected number of children from $x$ to $y$, for the discrete-time counterpart of a classical continuous-time BRW with parameter $\lambda$ and rate matrix $K$, is $m_{xy}=m_{xy}(\lambda)=\lambda k_{xy}$. The first moment matrix $M$ of this discrete-time counterpart equals $\lambda K$.
Thanks to the associated discrete-time counterpart, we extend the notion of irreducibility to 
 continuous-time process. More precisely, a continuous-time process is irreducible (by definition) if and only if its discrete-time counterpart is
 irreducible.
\\
  The local critical parameter $\lambda_s$ can be characterized in terms of the first-moment matrix and of the
   first arrival generating function $\Phi_\lambda$.
   We define 
$\Phi_\lambda(x,y|t):=\sum_{n =1}^\infty \varphi_{xy}^{(n)} t^n$, where
\[\begin{cases}
\displaystyle \varphi^{(n+1)}_{xy}
:=\sum_{w \in X \setminus\{y\}} m_{x w} \varphi^{(n)}_{w y}, \, \forall n \in \mathbb{N}, n \ge 1, \\
\displaystyle      \varphi^{(0)}_{wz}:=0, \, \varphi^{(1)}_{wz}:=m_{wz}, \, \forall w,z \in X.\\
\end{cases}
\]
Roughly speaking, $\varphi^{(n)}_{xy}$ is the expected number of particles alive at $y$ at time $n$,
when the initial state is just one particle at $x$ and the
process behaves like a BRW except that every particle reaching
$y$ at any time $i <n$ is immediately killed (before breeding).
  \begin{Theorem}\label{th:lambdas}
  (\cite[Theorem 4.1]{cf:BZ2})
  Given a continuous-time BRW $(X,K)$, let $k^{(n)}(\cdot,\cdot)$ and  $m^{(n)}(\cdot,\cdot)$ be the elements of the matrices $K^n$ and $M^n$, respectively.
  The local critical parameter satisfies
  \begin{equation}\label{eq:lambdas1}
  \begin{split}
  \lambda_s(x)&=1/\limsup_{n\to\infty}\sqrt[n]{k^{(n)}(x,x)}=
  \sup\{\lambda\in\R\colon \sum_n m_{xx}^{(n)}(\lambda)<+\infty\}\\
  & =
  \max\{ \lambda 
  %> 0
\in {\mathbb R}
\colon \Phi_\lambda(x,x|1)\leq 1\}=
\sup 
\{ \lambda 
%> 0
\in {\mathbb R}:
\colon\Phi_\lambda(x,x|1)< 1\}.
  \end{split}
  \end{equation}
  \end{Theorem}
The global critical parameter of the continuous-time BRW has a characterization in terms of solutions of certain inequalities
(see \cite[Theorem 4.2]{cf:BZ2}), which is unfortunately impractical to use in many cases.
In general $\lambda_w(x)$ depends on the whole offspring distributions and not only on the first-moment matrix $M$.
Theorem \ref{th:lambdaw} provides a lower bound in terms of $M$, which coincides with the actual
value of$\lambda_w(x)$ when
 the process satisfies some regularity
assumptions, such as transitivity or being an $\mathcal{F}$-BRW for every $\lambda>0$ (see for the definition
 \cite[Section 2.4]{cf:BZ17}).
  \begin{Theorem}\label{th:lambdaw}
  (\cite[Theorem 3.2]{cf:BZ14-SLS}, \cite[Theorem 4.1]{cf:Z1})
  Given a continuous-time BRW $(X,K)$. Then
  \begin{equation}\label{eq:lambdaw1}
  \lambda_w(x)\ge 1/\liminf_{n\to\infty}\sqrt[n]{\sum_{y\in X}k^{(n)}(x,y)}.
  \end{equation}
  Equality holds if $(X,K)$ is an $\mathcal F$-BRW.
  \end{Theorem}

\medskip

\begin{table}
	\begin{tabular}{ |c|c| }
		\hline 
		\rowcolor{lightgray}
		Notation & Meaning  \\ \hline
		$\lambda(x, A)$ & threshold between (a.s.) extinction and survival (w.p.p.) in $A$ \\
		\hline
		$\lambda_w(x)$ & threshold between (a.s.)  extinction and survival (w.p.p.) in $X$ \\ 
		%&   \\ 
		\hline
		$\lambda_s(x)$ & threshold between (a.s.)  extinction and survival (w.p.p.) in $x$   \\ 
		%&    \\ 
		\hline
	\end{tabular}
	\captionof{table}{Critical parameters. The BRW starts with one individual in $x\in X$. 
	The abbreviations a.s. and w.p.p. stand for \emph{almost sure} and \emph{with positive probability}, respectively.}\label{tb:tablelambda}
\end{table}

\section{Germ-monotone families of Branching Random Walks}
\label{sec:GMBRW}

 \subsection{Germ domination and critical parameters}
 \label{subsec:germd}

     Since the discrete-time counterpart of a 
  classical continuous-time BRW is indexed by $\lambda$ and shares its extinction probabilities, the expression 
  of $\lambda(x,A)$ in \eqref{eq:criticalcts} can be equivalently evaluated viewing the extinction probabilities $ \mathbf{q}(x,A|\lambda)$
  as those of the continuous-time process or of the discrete-time one. In particular we say that $\lambda(x,A)$ is a critical parameter
  of the discrete-time counterpart as well. 
  This motivates us to extend the definition of critical parameters to more general families
  of discrete-time BRWs indexed by a real parameter $\lambda$. 
  The main feature of the critical parameters is the fact that if $\lambda $ is smaller than the critical parameter, then there is almost sure 
  extinction, and if it is larger, then there is survival with positive probability. In the classical case this is
   due to the monotonicity in $\lambda$ of these processes.
  More precisely, the stochastic monotonicity of the family implies that the functions 
  $\lambda \mapsto \mathbf{q}(x,A|\lambda)$ are monotone and nonincreasing. It turns out that stochastic monotonicity is not needed,
  a weaker monotonicity, germ-monotonicity, suffices to define the critical parameters that play the role of thresholds between extinction and survival.
  The following definition recalls different notions of order for BRWs: the usual stochastic order $\succeq$ and two weaker order relations, the pgf order and the germ order,
  which are based on the  generating functions. These relations have first been introduced in
  \cite{cf:Hut2022} for translation invariant BRWs and then extended to general BRWs in \cite{cf:BZgerm}.

\begin{Definition}
	\label{def:ordering}
	Let $\boldmu:=\{\mu_x\}_{x \in X}$ and $\boldnu:=\{\nu_x\}_{x \in X}$ be two families of measures on 
	%$\mathbb{R}^X$ with support on 
	$S_X$.
	Let $G_{\boldmu}$ and $G_\boldnu$ be the associated generating functions.
	\begin{enumerate}
		\item 
		We say that $\boldmu$ \emph{stochastically dominates} $\boldnu$,
		$\boldmu \succeq \boldnu$ if and only if 
		$\mu_x \succeq \nu_x$ for all $x \in X$, that is, if and only if
		given a non-decreasing measurable function $F\colon S_X\to \R$, we have
		%for every $\mathbb{R}^X$-valued nondecreasing measurable function $F$ we have
		$\int F \diff \mu_x \ge \int F \diff \nu_x$ for all $x \in X$ such that the integrals are well defined.
		\item We say that $\boldmu$ \emph{pgf dominates} $\boldnu$,
		$\boldmu \gepgf \boldnu$ if and only if
		$G_\boldmu(\mathbf{z}) \le G_\boldnu(\mathbf{z})$ for all $\mathbf{z} \in [0,1]^X$.
		\item We say that $\boldmu$ \emph{germ dominates} $\boldnu$, 
		 $\boldmu \gegerm \boldnu$ if and only if there exists $\delta \in [0,1)$
		$G_\boldmu(\mathbf{z}) \le G_\boldnu(\mathbf{z})$ for all $\mathbf{z} \in [\delta,1]^X$.
	\end{enumerate}
	If $\# X=1$, that is, $\boldmu=\{\mu\}$ and $\boldnu=\{\nu\}$,
	%contains just one measure each, say $\mu$ and $\nu$ 
	then we simply write $\mu \gepgf \nu$ and $\mu \gegerm \nu$.
\end{Definition}
We observe that $\boldmu \succeq \boldnu \ \Rightarrow \boldmu \gepgf \boldnu \ \Rightarrow \ \boldmu \gegerm \boldnu$, but the reverse implications do not hold. 
%%%%%%%% parte che era Remark sul coupling e ordine stocastico
Clearly $G_\boldmu(\mathbf{z}) \le G_\boldnu(\mathbf{z})$ if and only if $G_\boldmu(\mathbf{z}|y) \le G_\boldnu(\mathbf{z}|x)$ for all $x \in X$; thus,
$\boldmu \gegerm \boldnu$ (with a certain $\delta <1$ ) if and only if $\mu_x \gegerm \nu_x$ for all $x \in X$ (with $\delta_x$ such that $\sup_{x\in X}\delta_x\le \delta<1$).

We employ the germ order of BRWs in the following definition.

\begin{Definition}
	\label{def:onepardependentfamily}
	Consider a family of measures $\Mfrak:=\{\mu_{x,\lambda}\}_{x \in X, \lambda >0}$ on $S_X$ such that if $\lambda\ge\lambda^*$, then $\boldmu_\lambda \gegerm \boldmu_{\lambda^*}$ (where $\boldmu_\lambda:=\{\mu_{x,\lambda}\}_{x \in X}\}$).  The family of BRWs $\{(X,\boldmu_\lambda)\}_{\lambda>0}$ is called 
	%\emph{germ-monotone family of BRWs} or simply 
	\emph{germ-monotone BRW} (in short, \emph{GMBRW}) and denoted by $(X, \Mfrak)$. Given $\lambda>0$, the $\lambda$-BRW associated to the GMBRW is denoted by $(X,\boldmu_\lambda)$.
	
	A GMBRW $(X, \Mfrak)$ is called \emph{regular} if and only if for all $x \in X$ the irreducible class $[x]$ induced by $(X, \boldmu_\lambda)$ does not depend on $\lambda>0$. 
	In particular, a GMBRW is called \emph{irreducible} if and only if $(X,\boldmu_\lambda)$ is irreducible for all $\lambda >0$ (in this case it is obviously regular). 
\end{Definition}

We note that it is possible to find germ-monotone BRWs which are not stochastically monotone, see Example \ref{ex:nonstochmon}.
Since the extinction probabilities depend on $\lambda$, given  a GMBRW $(X,\Mfrak)$ and a fixed $\lambda>0$, we denote by $\mathbf{q}^{\Mfrak}(A|\lambda)$ (or $\mathbf{q}^{\boldmu_\lambda}(A)$) the extinction probability vector for the BRW $(X,\boldmu_\lambda)$. When the family $\Mfrak$ is clear from the context, we simply write $\mathbf{q}(A|\lambda)$.
We are now able to define the critical parameters of a GMBRW.

\begin{Definition}\label{def:critGM}
Given a GMBRW $(X,\Mfrak)$,
$x \in X$ and $A \subseteq X$, we 
defined 
\begin{equation}\label{eq:criticalparametersA}
		\lambda^\Mfrak(x,A):=
		\inf \{\lambda>0\colon \,
		\mathbf{q}^\Mfrak(x,A|\lambda)<1 \}.
\end{equation}
When it is not important to emphasize the GMBRW $(X,\Mfrak)$ then we simply write $\lambda(x,A)$ instead of $\lambda^\Mfrak(x,A)$.
\end{Definition}

We observe that  in  general $\lambda \mapsto \mathbf{q}^\Mfrak(x,A|\lambda)$  does not need to be a monotone function.
Thus, the fact that $\lambda^\Mfrak(x,A)$ is a threshold between extinction and survival is not straightforward.
However, it is a consequence of the fact that if $\boldmu\gegerm\boldnu$, then extinction for $(X,\boldmu)$ implies
 extinction for $(X,\boldnu)$ 
 (this was proven in \cite[Theorem 1.3]{cf:Hut2022} for a restricted class of BRWs and extended to the most general case in \cite[Theorem 4.1]{cf:BZgerm}).
\begin{Proposition}\label{pro:criticalGM}
Given a GMBRW $(X,\Mfrak)$, for all
$x \in X$ and $A \subseteq X$,  $\lambda^\Mfrak(x,A):=
\sup \{\lambda>0\colon \,
\mathbf{q}^\Mfrak(x,A|\lambda)=1 \}$.
Moreover,  if $\boldmu_\lambda \gepgf \boldmu_{\lambda^*}$ (or $\boldmu_\lambda \succeq \boldmu_{\lambda^*}$), for all 
$\lambda\ge\lambda^*$, then the map $\lambda \mapsto \mathbf{q}(x,A|\lambda)$ is nonincreasing.
\end{Proposition}
\begin{proof}
By \eqref{eq:criticalparametersA}, for all $\lambda<\lambda^\Mfrak(x,A)$, we have $\mathbf{q}^\Mfrak(x,A|\lambda)=1$.
Moreover, for all $\lambda>\lambda^\Mfrak(x,A)$, there exists $\lambda^*\in (\lambda^\Mfrak(x,A),\lambda)$ such that
$\mathbf{q}^\Mfrak(x,A|\lambda^*)<1$.
On the other hand, by definition of GMBRW, $\boldmu_\lambda \gegerm \boldmu_{\lambda^*}$.
It was proven in \cite[Theorem 4.1(ii)]{cf:BZgerm} that by germ-domination, if $\mathbf{q}^\Mfrak(x,A|\lambda^*)<1$, then
also $\mathbf{q}^\Mfrak(x,A|\lambda)<1$. This means that there is a positive probability of survival for all $\lambda>\lambda^\Mfrak(x,A)$,
or, in other words, that$\lambda^\Mfrak(x,A)$ is a threshold between survival and extinction in $A$ starting from $x$ and we can
alternatively write
 $\lambda^\Mfrak(x,A):=
\sup \{\lambda>0\colon \,
\mathbf{q}^\Mfrak(x,A|\lambda)=1 \}$.
\\
In the case when $\boldmu_\lambda \gepgf \boldmu_{\lambda^*}$, then 
by \cite[Theorem 4.1(i)]{cf:BZgerm}, we have that $\mathbf{q}^\Mfrak(x,A|\lambda^*)\le\mathbf{q}^\Mfrak(x,A|\lambda)$, which means that
$\lambda \mapsto \mathbf{q}(x,A|\lambda)$ is nonincreasing.
\end{proof}
As in the case of continuous-time BRWs, we define the global and local critical parameters for GMBRWs:
\begin{equation}\label{eq:criticalparameters}
	\begin{split}
		\lambda_w^\Mfrak(x)&:=\lambda^\Mfrak(x,X);\qquad 
		%\pr^{\delta_{x}}\left(\exists t\colon \eta_t=\mathbf{0}\right)<1\}\\
		\lambda_s^\Mfrak(x):=\lambda^\Mfrak(x, \{x\}).
		%\pr^{\delta_{x}}\left(\exists \bar t\colon \eta_t(x)=0,\,\forall t\ge\bar t\right)<1\},
	\end{split}
\end{equation}
Table \ref{tb:table3} summarizes the results of Proposition \ref{pro:criticalGM}. Note that the behaviour in $A$ (extinction or survival),
when $\lambda=\lambda(x,A)$ depends in general on the process. 
It is known that there is 
extinction in $x$ for $\lambda=\lambda_s(x)$, while there are examples where there is global survival at $\lambda_w(x)$ (see 
\cite[Theorem~4.7 and Example~3]{cf:BZ2} and \cite[Theorem 4.1]{cf:BZgerm}).

\begin{table}
	\begin{tabular}{ |c|c|}
		\hline 
		\rowcolor{lightgray}
		%Relations between c
		Parameters and probabilities  & Probabilities and parameters \\ \hline
		%&   \\  
%		$\lambda \le \lambda_s(x) \Longleftrightarrow {\mathbf{q}}(x,\{x\}|\lambda)= 1$ \\ 
%		%&    \\ 
%		\hline
%		$\lambda > \lambda_s(x) \Longleftrightarrow {\mathbf{q}}(x,\{x\}|\lambda)< 1$ \\ 
%		%&    \\ 
%		\hline
		$\lambda < \lambda(x,A) \Longrightarrow {\mathbf{q}}(x,A|\lambda)= 1$
		%, \quad
		& ${\mathbf{q}}(x,A|\lambda)= 1 \Longrightarrow \lambda \le \lambda(x,A)$   \\ 
		%&    \\ 
		\hline
		$\lambda > \lambda(x,A) \Longrightarrow {\mathbf{q}}(x,A|\lambda)< 1$
		%, \quad
		& ${\mathbf{q}}(x,A|\lambda)< 1 \Longrightarrow \lambda \ge \lambda(x,A)$ \\
		%&    \\ 
		\hline
	\end{tabular}
	\captionof{table}{
	%Relations between c
	Critical parameters and extinction probabilities.
	%; see \cite[Theorem~4.7 and Example~3]{cf:BZ2} and \cite[Theorem 4.1]{cf:BZgerm}.
	}	\label{tb:table3}
\end{table}

Henceforth, unless otherwise explicitly mentioned, we consider only regular, irreducible GMBRW.
For these GMBRW, the critical parameters do not depend on $x$ and we write $\lambda^\Mfrak(A)$ (or $\lambda(A)$ when there is no ambiguity on the process). Indeed, given $x,y\in X$, if the GMBRW is regular and irreducible, there is a
positive probability that an individual living in $x$ has at least one descendant in $y$. By markovianity, a
positive probability of survival in a set $A$ starting from $y$ implies the same starting from $x$.
Swapping the roles of $x$ and $y$, we obtain $\lambda^\Mfrak(x,A) = \lambda^\Mfrak(y,A)$.
Moreover, for all finite nonempty $A \subseteq X$,
$\lambda^\Mfrak(A)=\lambda_s^\Mfrak$.
  This follows from the fact that $\mathbf{q}^\Mfrak(x,A|\lambda)$  is constant for all finite nonempty $A$. 
  % Moreover, in this case, for every nonempty $B \subseteq X$ we have $\lambda_w \le \lambda(B) \le \lambda_s$.
  \\
  Note that if  $A\subseteq B$, then $A \stackrel{\boldnu}{\Rightarrow} B$ for all $\boldnu$,
and this implies that $\lambda^\Mfrak(A) \ge \lambda^\Mfrak(B)$ for all GMBRW $(X,\Mfrak)$.
Therefore in the set $\{\lambda^\Mfrak(A)\}_{A\subseteq X, A\neq\emptyset}$, the smallest critical value is $\lambda_w^\Mfrak$, while the maximum is  $\lambda_s^\Mfrak$.

We recall that
Definition \ref{def:equivalentsets} was given for single discrete-time BRWs, but it can be extended to families of BRWs, such as GMBRWs. Given $\Mfrak=\{\boldmu_\lambda\}_{\lambda>0}$, we say that $A \stackrel{\Mfrak}{\Rightarrow} B$ if  $A \stackrel{\boldmu_\lambda}{\Rightarrow} B$ for every $\lambda>0$.
Similarly, $A \stackrel{\Mfrak}{\nRightarrow} B$ means that there exists $\lambda$ such that $A \stackrel{\boldmu_\lambda}{\nRightarrow} B$.

In Theorem\ref{th:lambdas} we recalled the characterizations of $\lambda_s(x)$ for classical continuous-time BRWs.
Now we characterize $\lambda_s^\Mfrak(x)$ for GMBRWs and provide a lower bound for $\lambda^\Mfrak(x,A)$.

\begin{Proposition}
	\label{pro:criticalA}
	Let $(X, \Mfrak)$ be a GMBRW and $A \subseteq X$. 
	For every $\lambda>0$ denote by $m_{xy}^{(n)}(\lambda)$ the entries of the $n$th power matrix $M_\lambda^n$ where  $M_\lambda=(m_{xy}(\lambda))_{x,y \in X}$ is the first moment matrix of $(X, \boldmu_\lambda)$.
	Then
	\[
	 \begin{split}
  \lambda_s(x)&=  \sup\{\lambda\in\R\colon \sum_n m_{xx}^{(n)}(\lambda)<+\infty\}\\
	\lambda^\Mfrak(x,A)& \ge\sup\{\lambda\in\R\colon \sum_{n\in\mathbb{N}}
	 \sum_{y \in A} m_{xy}^{(n)}(\lambda)<+\infty\}.
	 \end{split}
	\]
	
\end{Proposition}
\begin{proof}
	Observe that $\sum_{y \in A} m_{xy}^{(n)}(\lambda)$ represents the expected number of individuals alive at time $n$ in $A$, given that the process starts from a single particle at $x$. Consequently, 
	$\sum_{n\in\mathbb{N}} \sum_{y \in A} m_{xy}^{(n)}(\lambda)<+\infty
	$
	means that the expected total number of descendants that have ever lived in $A$ is finite. Hence, the total number of descendants that have ever lived in $A$ is a.s.~finite; thus $\mathbf{q}(x,A|\lambda)=1$ and, according to Table~\ref{tb:table3}, $\lambda \le \lambda(x,A)$.
\end{proof}

\subsection{Continuous-time BRWs and ageing BRWS}
\label{subsec:continuous}

The discrete-time counterpart of a continuous-time BRW is a natural example of
a regular GMBRW.
Indeed, if we compute the total number of offspring of a particle at $x$, and observe 
that each child is placed independently in $y$ with probability ${k_{xy}}/{\sum_{y \in X} k_{xy}}$, we have that
\begin{equation}\label{eq:mucontinuous}
	\mu_{x,\lambda}(f)= \frac{1}{\big (1+\lambda \sum_{y \in X} k_{xy} \big )} \frac{(\sum_{y \in X} f(y))!}{\prod_{y \in X} f(y)!} \prod_{y \in X} \Big ( \frac{\lambda k_{xy} }{1+\lambda \sum_{y \in X} k_{xy}}\Big )^{f(y)},
\end{equation}
and
\[
G_{\boldmu_\lambda}(\mathbf{z}|x)=\frac{1}{1+\lambda \sum_{y \in X} k_{xy}(1-\mathbf{z}(y))}
\]
for all $\mathbf{z}\in [0,1]^X$ and $x \in X$.
If $\lambda \ge \lambda^*$, then $G_{\boldmu_\lambda}(\mathbf{z})\le G_{\boldmu_\lambda^*}(\mathbf{z})$
for all $\mathbf z\in [0,1]^X$, thus $\boldmu_{\lambda}\gepgf\boldmu_{\lambda^*}$.
The regularity follows from the fact that the irreducible class of $[x]$ is the usual irreducible class induced by the matrix $K$, which does not depend on $\lambda>0$.

From  Proposition~\ref{pro:criticalA} we get a lower bound for $\lambda(x,A)$ in the case of the  continuous-time BRW.

\begin{Corollary}
	\label{cor:criticalA}
	Let $(X, K)$ be a continuous-time BRW and $A \subseteq X$. 
	Then $\lambda(x,A) \ge 1/\limsup_{n \in \mathbb{N}} \sqrt[n]{\sum_{y \in A} k_{xy}^{(n)}}$.
\end{Corollary} 		
\begin{proof}
	Since the extinction probabilities of the continuous-time process coincide with those of its discrete-time counterpart, we apply Proposition~\ref{pro:criticalA} to the latter. The claim follows from the identity 
	$m_{xy}^{(n)}(\lambda)=\lambda^n k_{xy}^{(n)}$ for all $n \in \mathbb{N}$ and $x,y \in X$.
	In particular, if 
	$
	\lambda < {1}/{\limsup_{n \in \mathbb{N}} \sqrt[n]{\sum_{y \in A} k_{xy}^{(n)}}}
	$,
	then 
	$
	\sum_{n\in\mathbb{N}} \lambda^n \sum_{y \in A} k_{xy}^{(n)} < +\infty$,
	which in turn implies $\lambda \le \lambda(x,A)$.
\end{proof}
A comparison with Theorems \ref{th:lambdas} and \ref{th:lambdaw} shows that Corollary \ref{cor:criticalA}, when applied to the
cases $A=\{x\}$ and $A=X$, identifies the best lower bound for $\lambda_s(x)$, while it misses the best one for
$\lambda_w$, when $\limsup_{n \in \mathbb{N}} \sqrt[n]{\sum_{y \in X} k_{xy}^{(n)}}>
\liminf_{n \in \mathbb{N}} \sqrt[n]{\sum_{y \in X} k_{xy}^{(n)}}$.

A more realistic continuous-time model is given by the \emph{ageing BRW}, in which an individual’s reproductive ability depends on its age (see, for instance, \cite{cf:BZageing}). Let $\mathcal{R}:=\{r_{xy}\}_{x,y \in X}$ be a family of nonnegative measurable functions such that $\sum_{y \in X} r_{xy} \in L^1_{\mathrm{loc}}([0,+\infty))$ for every $x \in X$. For each $x \in X$, we assume that the lifetime of a particle located at $x$ is governed by a law with cumulative distribution function $T_x$, where $\{T_x\}_{x \in X}$ is a collection of nondecreasing, right-continuous functions satisfying $T_x(0)=0$ and $\lim_{t \to +\infty} T_x(t)=1$.

The dynamics of the process are as follows. When a particle is born at site $x$ at time $\bar t$, a family of independent
inhomogeneous Poisson point process with intensities $\{\lambda r_{xy}(\cdot)\}_{y \in X}$ is activated. During the random time interval $[\bar t, \bar t+\hat t]$, where $\hat t$ denotes the random lifetime of the particle with cumulative distribution function $T_x$, at each arrival time of a Poisson point process with intensity $t \mapsto \lambda r_{xy}(t-\bar t)$ a new particle is placed at $y$. All Poisson point processes and lifetimes, depending on $x,y \in X$ and on the individual particle, are assumed to be independent. In general, this process is non-Markovian, unless the intensities are constant and lifetimes are exponentially distributed. Nonetheless, the associated discrete-time process describing the number and locations of the offspring at the end of an individual’s lifetime is a regular GMBRW.
It is possible to explicitly compute the measures of the family $\boldmu_\lambda$ as

\begin{equation}\label{eq:explicitmu}
	\mu_{x,\lambda}(f)=\int_0^\infty \prod_{y \in X} \Big (
	\frac{\exp \big (-\lambda \int_0^t r_{xy}(s) \diff s \big ) \big (\lambda \int_0^t r_{xy}(s) \diff s \big )^{f(y)}}{f(y)!}  \Big )\pr_{T_x}(\diff t)
\end{equation}
where $\pr_{T_x}$ is the law whose c.d.f.~is $T_x$.
The corresponding generating function is
\[
G_{\boldmu_\lambda}(\mathbf{z}|x)=
\int_0^\infty 
\exp \Big (-\lambda \int_0^t \sum_{y \in X} r_{xy}(s) (1-\mathbf{z}(y)) \diff s \Big )  \pr_{T_x}(\diff t).
\]

If 
$r_{xy}(t)=k_{xy}$ is a constant function and the lifetime is exponentially distributed with parameter $1$ (i.e.~$T_x(t):=\ident_{[0,+\infty)]}(t)(1- \exp(-t))$ then equation~\eqref{eq:explicitmu} becomes equation~\eqref{eq:mucontinuous} and we have the usual continuous-time BRW.

Since the discrete-time counterpart of an ageing BRW is a GMBRW, it is possible to define its critical parameters 
$\lambda(x,A)$, for which the lower bound in Corollary \ref{cor:criticalA} hold, with
\[
k_{xy}=\int_0^\infty\int_0^t r_{xy}(s) \diff s \pr_{T_x}(\diff t).
\]
This shows that dealing with GMBRWs makes it possible to derive results for a broader class of processes, 
including some non-Markovian processes.

\section{Critical parameters for modified BRWs and GMBRWs}
\label{sec:survivalprob}

In this section, we examine how modifications 
of the reproduction laws in some vertices
may 
affect extinction probabilities and critical parameters.
Since 
survival on a fixed set for a continuous-time BRW is equivalent to 
survival on the same set for its discrete-time counterpart, it is 
natural to use results from the discrete-time case to infer 
corresponding results for continuous-time BRWs.

The main technical tool of this section is Theorem~\ref{cor:equivalence}, from which we derive Theorem~\ref{th:modifiedBRW}, Proposition~\ref{pro:pureweak-nonstrong} and
Theorem~\ref{th:mainmod}.
Given  two BRWs
$(X,\boldmu)$ and $(X,\boldnu)$, we
denote by ${\mathbf{q}}^{\boldmu}$ and ${\mathbf{q}}^{\boldnu}$ their respective extinction probability vectors. Similarly, we write $\stackrel{\boldmu}{\Rightarrow}$ and $\stackrel{\boldnu}{\Rightarrow}$ for the survival implications introduced in Section~\ref{subsec:equivalentsets}. The following result is a substantial generalization of \cite[Theorem~3]{cf:BZ2025} and serves as a cornerstone for the remainder of the paper.

\begin{Theorem}\label{th:modifiedBRW}
	Let $(X,\boldmu)$ and $(X,\boldnu)$ be two BRWs and define  
	$\Delta:=\{
	x \in X \colon \mu_x\neq \nu_x \}$. 
	If $A \subseteq X$ is such that $\Delta \stackrel{\boldmu}{\Rightarrow} A$ and $\Delta \stackrel{\boldnu}{\Rightarrow} A$, then for all $B\subseteq X$ 
	\begin{equation}\label{eq:thm3.1}
		{\mathbf{q}}^\boldmu(A) \leq {\mathbf{q}}^\boldmu(B) \Longleftrightarrow
		{\mathbf{q}}^\boldnu(A) \leq {\mathbf{q}}^\boldnu(B).
	\end{equation}
	If, in addition,  $\Delta \stackrel{\boldmu}{\Rightarrow} B$ and $\Delta \stackrel{\boldnu}{\Rightarrow} B$, then
	\begin{equation}\label{eq:thm3.2}
		\begin{cases}
			{\mathbf{q}}^\boldmu(A) = {\mathbf{q}}^\boldmu(B) &\Longleftrightarrow
			{\mathbf{q}}^\boldnu(A) = {\mathbf{q}}^\boldnu(B),\\    
			{\mathbf{q}}^\boldmu(A) < {\mathbf{q}}^\boldmu(B) &\Longleftrightarrow
			{\mathbf{q}}^\boldnu(A) < {\mathbf{q}}^\boldnu(B),\\    
			{\mathbf{q}}^\boldmu(A) > {\mathbf{q}}^\boldmu(B) &\Longleftrightarrow
			{\mathbf{q}}^\boldnu(A) > {\mathbf{q}}^\boldnu(B);
		\end{cases}
	\end{equation}
	moreover $\mathbf{q}^{\boldmu}(A)$ and $\mathbf{q}^{\boldmu}(B)$ are not comparable if and only if $\mathbf{q}^{\boldnu}(A)$ and $\mathbf{q}^{\boldnu}(B)$ are not comparable.
	
	Finally, if $(X, \boldmu)$ and $(X,\boldnu)$ are both irreducible and $\Delta$ is finite, then  
	\eqref{eq:thm3.2} holds for all nonempty subsets $A, B \subseteq X$.
\end{Theorem}

\begin{proof}%[Proof of Theorem~\ref{th:modifiedBRW}]
	By Remark~\ref{rem:1}, ${\mathbf{q}}^\boldmu(A)={\mathbf{q}}^\boldmu(A \cup \Delta)$ and
	${\mathbf{q}}^\boldnu(A)={\mathbf{q}}^\boldnu(A \cup \Delta)$; hence it suffices to establish
	 	\eqref{eq:thm3.1} in the case $A \supseteq \Delta$, which was carried out in
	\cite[Theorem 3]{cf:BZ2025}. This completes the proof of  
	\eqref{eq:thm3.1}.
	
	We turn to the first line of  
	\eqref{eq:thm3.2}, namely
	\begin{equation}
		\label{eq:thm3.2bis}
		{\mathbf{q}}^\boldmu(A) = {\mathbf{q}}^\boldmu(B) \Longleftrightarrow
		{\mathbf{q}}^\boldnu(A) = {\mathbf{q}}^\boldnu(B).
	\end{equation}
	This equivalence is obtained by applying  \eqref{eq:thm3.1} twice,
	the second time with $A$ and $B$ interchanged. The remaining lines of
	 \eqref{eq:thm3.2} then follow easily from  \eqref{eq:thm3.2bis}
	in conjunction with  \eqref{eq:thm3.1}. Indeed, we have proven that
	\[
	\begin{cases}
		{\mathbf{q}}^\boldmu(A) \leq {\mathbf{q}}^\boldmu(B)\\
		{\mathbf{q}}^\boldmu(A) \neq {\mathbf{q}}^\boldmu(B)
	\end{cases}
	\Longleftrightarrow \ \
	\begin{cases}
		{\mathbf{q}}^\boldnu(A) \leq {\mathbf{q}}^\boldnu(B)\\
		{\mathbf{q}}^\boldnu(A) \neq {\mathbf{q}}^\boldnu(B)
	\end{cases}
	\]
	which is precisely the second line of  \eqref{eq:thm3.2}, while the third line
	is obtained by interchanging the roles of $A$ and $B$. Finally,
	 \eqref{eq:thm3.2} yields that $\mathbf{q}^{\boldmu}(A)$ and
	$\mathbf{q}^{\boldmu}(B)$ are not comparable if and only if $\mathbf{q}^{\boldnu}(A)$
	and $\mathbf{q}^{\boldnu}(B)$ are not comparable.
	
	By Remark~\ref{rem:1}, if both $(X, \boldmu)$ and $(X,\boldnu)$ are irreducible and
	$\Delta$ is finite, then $\Delta \stackrel{\boldmu}{\Rightarrow} A$,
	$\Delta \stackrel{\boldmu}{\Rightarrow} B$, $\Delta \stackrel{\boldnu}{\Rightarrow} A$, and
	$\Delta \stackrel{\boldnu}{\Rightarrow} B$ for all nonempty subsets $A, B$; consequently,
	 \eqref{eq:thm3.2} holds for all nonempty subsets $A, B \subseteq X$.
\end{proof}

While Theorem \ref{th:modifiedBRW} deals with the modification of a single BRW, one can consider modifications
of families of BRWs. Indeed, given two irreducible GMBRWs $(X,\Mfrak)$ and $(X,\Nfrak)$, 
we denote by $\Delta_{\Mfrak, \Nfrak}:=\{x \in X \colon \mu_{x,\lambda} \neq \nu_{x, \lambda}, \, \textrm{ for some } \lambda >0\}$ the set where $\Mfrak$ and $\Nfrak$ are different.
When $\Delta_{\Mfrak, \Nfrak}$ is finite, we say that $\Nfrak$ is a \emph{local modification} of $\Mfrak$ and write
$\Mfrak \sim \Nfrak$.
It is easy to see that the  relation $\sim$ is an equivalence relation on the family of irreducible GMBRWs on $X$.
We denote by $[\Mfrak]$ the equivalence class of $\Mfrak$.
\\
We now extend Theorem \ref{th:modifiedBRW} to modifications of GMBRWs.

\begin{Theorem}
	\label{th:modifiedBRW-CT}
	Let $(X, \Mfrak)$ and $(X,\Nfrak)$ be two irreducible GMBRWs such that 
	$\Mfrak \sim\Nfrak$; 
	 then for all nonempty subsets $A,B \subseteq X$ and for all $\lambda>0$ we have
	%\begin{equation}\label{eq:thm3.2}
	\[
	\begin{cases}
		{\mathbf{q}}^{\Mfrak}(A|\lambda) = {\mathbf{q}}^{\Mfrak}(B|\lambda) &\Longleftrightarrow
		{\mathbf{q}}^{\Nfrak}(A|\lambda) = {\mathbf{q}}^{\Nfrak}(B|\lambda),\\    
		{\mathbf{q}}^{\Mfrak}(A|\lambda) < {\mathbf{q}}^{\Mfrak}(B|\lambda) &\Longleftrightarrow
		{\mathbf{q}}^{\Nfrak}(A|\lambda) < {\mathbf{q}}^{\Nfrak}(B|\lambda),\\    
		{\mathbf{q}}^{\Mfrak}(A|\lambda) > {\mathbf{q}}^{\Mfrak}(B|\lambda) &\Longleftrightarrow
		{\mathbf{q}}^{\Nfrak}(A|\lambda) > {\mathbf{q}}^{\Nfrak}(B|\lambda),
	\end{cases}
	\]
	%\end{equation}
	moreover $\mathbf{q}^{\Mfrak}(A|\lambda)$ and $\mathbf{q}^{\Mfrak}(B|\lambda)$ are not comparable if and only if $\mathbf{q}^{\Nfrak}(A|\lambda)$ and $\mathbf{q}^{\Nfrak}(B|\lambda)$ are not comparable.
\end{Theorem}

In the results of this section we assume 
that the following assumption holds.
%one of the following.
\begin{Assumption}
	\label{assump:1}
Let $(X,\Mfrak)$, $(X,\Nfrak)$ be two irreducible GMBRWs  and let $A, B \subseteq X$
	two nonempty subsets    such that  $\Delta_{\Mfrak, \Nfrak} \stackrel{\Mfrak}{\Rightarrow} A$, $\Delta_{\Mfrak, \Nfrak} \stackrel{\Mfrak}{\Rightarrow} B$, $\Delta_{\Mfrak, \Nfrak} \stackrel{\Nfrak}{\Rightarrow} A$ and $\Delta_{\Mfrak, \Nfrak} \stackrel{\Nfrak}{\Rightarrow} B$.
\end{Assumption}

Examples~\ref{ex:ass1} and~\ref{ex:ass2} provide natural settings in which 
Assumption~\ref{assump:1} is satisfied. In particular, the case of 
Example~\ref{ex:ass1} will be used frequently in what follows.
\begin{Example}\label{ex:ass1}
	Consider two irreducible GMBRWs $(X,\Mfrak)$ and $(X,\Nfrak)$ such that  $\Mfrak \sim \Nfrak$.
	Then Assumption \ref{assump:1} holds for all couples of nonempty sets $A, B \subseteq X$.
	Indeed, by Remark~\ref{rem:1}, if $(X,\Mfrak)$ is an irreducible GMBRW and $\Delta\subseteq X$ is finite and nonempty,
	then $\Delta \stackrel{\boldmu_\lambda}{\Rightarrow} A$ for all $A\subseteq X$, $A\neq\emptyset$, $\lambda>0$. 
	%Consequently, if $\Mfrak \calR \Nfrak$, then $\Delta_{\Mfrak,\Nfrak} \Rightarrow A$ for every nonempty subset $A \subseteq X$. 
\end{Example}
\begin{Example}\label{ex:ass2}
	Consider two irreducible GMBRWs $(X,\Mfrak)$ and $(X,\Nfrak)$ such that $A:=\Delta_{\Mfrak,\Nfrak}\neq\emptyset$.
	Take $B\subseteq X$ satisfying $A\subseteq B$.
	Then Assumption \ref{assump:1} holds.
\end{Example}

Theorem \ref{th:modifiedBRW}, applied to a GMBRW, leads to Proposition~\ref{pro:pureweak-nonstrong} and Theorem~\ref{th:mainmod}, which describe how 
modifications of the GMBRW affect the critical parameters.
%phase diagram of a continuous-time BRW 
These results extend \cite[Corollary 2]{cf:BZ2025}. 
To improve readability, henceforth we often use the notation $\lambda(\cdot)$ and  $\lambda^*(\cdot)$
instead of $\lambda^\Mfrak(\cdot)$ and  $\lambda^\Nfrak(\cdot)$.
 \begin{Proposition}\label{pro:pureweak-nonstrong}
	Let $(X,\Mfrak)$ and $(X,\Nfrak)$ be two irreducible GMBRWs and $A, B \subseteq X$ satisfying  Assumption~\ref{assump:1}. Denote by $\lambda(\cdot)$ and $\lambda^*(\cdot)$ the critical parameters of $(X,\Mfrak)$ and $(X,\Nfrak)$ respectively.
	Then the following are equivalent:
	\begin{enumerate}
		\item $\min(\lambda^{{*}}(A), \lambda^{{*}}(B)) < \min(\lambda(A), \lambda(B)) $;
		\item $\max(\lambda^{{*}}(A), \lambda^{{*}}(B)) < \min( \lambda(A), \lambda(B))$;
		\item  $\lambda^{{*}}(A)=\lambda^{{*}}(B) < \min( \lambda(A) ,  \lambda(B))$.
	\end{enumerate}
\end{Proposition}
\begin{proof}%[Proof of Corollary~\ref{cor:pureweak-nonstrong}]
	%We start by proving the equivalence $(1) \Longleftrightarrow (3)$. 
	By Remark~\ref{rem:1},
	${\mathbf{q}}^\Mfrak(A|\lambda)={\mathbf{q}}^\Mfrak(A \cup \Delta_{\Mfrak, \Nfrak}|\lambda)$,
	${\mathbf{q}}^{\Nfrak}(A|\lambda)={\mathbf{q}}^{\Nfrak}(A \cup \Delta_{\Mfrak, \Nfrak}|\lambda)$,
	${\mathbf{q}}^\Mfrak(B|\lambda)={\mathbf{q}}^\Mfrak (B \cup \Delta_{\Mfrak, \Nfrak}|\lambda)$, and
	${\mathbf{q}}^{\Nfrak}(B|\lambda)={\mathbf{q}}^{\Nfrak}(B \cup \Delta_{\Mfrak, \Nfrak}|\lambda)$
	for all $\lambda>0$; consequently,
	$\lambda(A)=\lambda(A \cup \Delta_{\Mfrak, \Nfrak})$,
	$\lambda^*(A)=\lambda^*(A \cup \Delta_{\Mfrak, \Nfrak})$,
	$\lambda(B)=\lambda(B \cup \Delta_{\Mfrak, \Nfrak})$, and
	$\lambda^*(B)=\lambda^*(B \cup \Delta_{\Mfrak, \Nfrak})$.
	\\	
	Therefore, without loss of generality, we can assume that
	$\Delta_{\Mfrak, \Nfrak} \subseteq A, B$. The implications
	$(3) \Longrightarrow (2) \Longrightarrow (1)$ are immediate.
	We prove $(1) \Longrightarrow (3)$.
	Suppose that $\lambda^{{*}}(A) = \min(\lambda^{{*}}(A), \lambda^{{*}}(B))
	< \min(\lambda(A), \lambda(B))$, and pick
	$\lambda \in \big(\lambda^{{*}}(A),\, \min(\lambda(A), \lambda(B))\big)$.
	For such $\lambda$, we have $\mathbf{q}^{\Mfrak}(A|\lambda)=\mathbf{q}^{\Mfrak}(B|\lambda)=\mathbf{1}$;
	Theorem~\ref{th:modifiedBRW} then gives
	$\mathbf{q}^{\Nfrak}(B|\lambda)=\mathbf{q}^{\Nfrak}(A|\lambda)<\mathbf{1}$,
	where the strict inequality follows from $\lambda > \lambda^{{*}}(A)$.
	Since $\mathbf{q}^{\Nfrak}(B|\lambda)<\mathbf{1}$ implies $\lambda \ge \lambda^{{*}}(B)$,
	we deduce $\lambda^{{*}}(B) \le \lambda$ for all $\lambda^{{*}}(A) <\lambda< \min(\lambda(A),  \lambda(B))$.
	Combined with $\lambda^{{*}}(A) = \min(\lambda^{{*}}(A), \lambda^{{*}}(B))$,
	this yields $\lambda^{{*}}(A)=\lambda^{{*}}(B)$.
	The case $\lambda^{{*}}(B)=\min(\lambda^{{*}}(A), \lambda^{{*}}(B))$ is
	entirely analogous, with the roles of $A$ and $B$ interchanged.
\end{proof}

\begin{Theorem}\label{th:mainmod}
	Let 
	$(X,\Mfrak)$ be an irreducible GMBRW, with critical parameters denoted by $\lambda(\cdot)$.
	Let $A, B \subseteq X$ be such that $\lambda(A)<\lambda(B)$.
	Suppose that $(X,\Nfrak)$ is another irreducible GMBRW, with critical parameters denoted by $\lambda^*(\cdot)$, and that
	Assumption~\ref{assump:1} holds.
	Then
	\begin{enumerate}
		\item$\lambda^*(A)\le\lambda(A)$ and  $\lambda^*(A)\le\lambda^*(B)$.
		\item Either $\lambda^*(A)=\lambda^*(B)\le\lambda(A)$ or $\lambda^*(A)=\lambda(A)<\lambda^*(B)$.
	\end{enumerate}
\end{Theorem}
\begin{proof}
	\leavevmode
	\begin{enumerate}
		\item 
		We proceed by contradiction. By hypothesis, $\lambda(A) < \lambda(B)$. 
		If $\min(\lambda^{{*}}(A), \lambda^{{*}}(B)) > \lambda(A) = \min(\lambda(A), 
		\lambda(B))$, then the equivalence of $(1)$ and $(3)$ in 
		Proposition~\ref{pro:pureweak-nonstrong}, with the role of $\Mfrak$ and $\Nfrak$ interchanged, would imply $\lambda(A) = \lambda(B)$, 
		contradicting our assumption. Hence $\min(\lambda^{{*}}(A), \lambda^{{*}}(B)) 
		\le \lambda(A)$.
		
		\noindent To conclude that $\lambda^*(A) \le \lambda(A)$, it suffices to show that $\lambda^*(A)=\min(\lambda^{{*}}(A), \lambda^{{*}}(B))$, since in that case $\lambda^*(A) = 
		\min(\lambda^{{*}}(A), \lambda^{{*}}(B)) \le \lambda(A)$. Suppose by 
		contradiction that $\lambda^*(A) > \lambda^*(B)$, so that $\lambda^*(B) = 
		\min(\lambda^{{*}}(A), \lambda^{{*}}(B)) \le \lambda(A)$. Apply
		Proposition~\ref{pro:pureweak-nonstrong} again, this time with the roles of $A$ and $B$ interchanged: since (3) does not hold then (1) does not hold as well, whence
		$\lambda(A) = \min(\lambda(A), \lambda(B)) \le \lambda^*(B)$.
		We already proved that $\lambda^*(B)=\min(\lambda^{{*}}(A), \lambda^{{*}}(B)) 
		\le \lambda(A)$, therefore 
		$\lambda(A) = \lambda^*(B) < \min(\lambda(B), \lambda^*(A))$.
		
		\noindent Now, for any $\lambda \in \big(\lambda(A), \min(\lambda(B), 
		\lambda^*(A))\big)$, we have $q^\Mfrak(A|\lambda) < \mathbf{1} = 
		q^\Mfrak(B|\lambda)$ and $q^\Nfrak(B|\lambda) < \mathbf{1} = 
		q^\Nfrak(A|\lambda)$, which contradicts Theorem~\ref{th:modifiedBRW}. 
		We conclude that $\lambda^*(A) \le \lambda^*(B)$, as required.
		\item 
		We already proved in (1) that $\lambda^*(A)\le \min(\lambda(A),\lambda^*(B))$ is a consequence of $\lambda(A)<\lambda(B)$. If $\lambda^*(A)=\lambda^*(B)\le\lambda(A)$ does not hold then $\lambda^*(A)<\lambda^*(B)$ which, by;
		switching the roles of $\Mfrak$ and $\Nfrak$, implies now $\lambda^*(A)\le\lambda(A)$, thus $\lambda^*(A)=\lambda(A)$.
		Since by (1) we know that $\lambda^*(A) \le \lambda^*(B)$, then the only other possibility is  $\lambda^*(A)=\lambda^*(B)\le\lambda(A)$.
	\end{enumerate}
\end{proof}

It is worth noting that Proposition~\ref{pro:pureweak-nonstrong} and Theorem~\ref{th:mainmod} are equivalent. Indeed, on the one hand the proof of 
Theorem~\ref{th:mainmod} shows that it is a corollary of Proposition~\ref{pro:pureweak-nonstrong}; on the other hand, it is a simple exercise to show
 that Proposition~\ref{pro:pureweak-nonstrong} can be proven by using theorem~\ref{th:mainmod}.

%\textcolor{red}{
%As a consequence of Theorem~\ref{th:mainmod}, when $(X,\Mfrak)$, $(X,\Nfrak)$, $A$ and $B$ satisfy Assumption \ref{assump:1}
%and $\lambda(A)<\lambda(B)$, it is not possible that $\lambda^*(A)>\lambda^*$, nor 
%$\lambda^*(A)>\lambda(A)$. 
%Moreover, whenever  
%$(X,\Mfrak)$, $(X,\Nfrak)$, $A$ and $B$ satisfy Assumption \ref{assump:1}, 
% only some mutual positions for the critical parameters of $\Nfrak$ are possible.
%Namely, if we denote by $\lambda(\cdot)$ and $\lambda^*(\cdot)$ the critical parameters of $\Mfrak$ and $\Nfrak$ respectively,
%the following situations, and those obtained by switching $A$ and $B$ or  $\Mfrak$ and $\Nfrak$, are the only ones possible:
%\begin{enumerate}
%	%	\item $\lambda(A)=\lambda(B) < \min(\lambda^*(A), \lambda^*(B))$,
%	%	\item $\lambda(A)=\lambda^*(A)<\min(\lambda(B), \lambda^*(B))$,
%	%	\item $\lambda(A)=\lambda(B) = \min(\lambda^*(A), \lambda^*(B))$,
%	%	\item $\lambda(A)=\lambda^*(A) = \lambda^*(B)<\lambda(B)$.
%	\item $\lambda^*(A)=\lambda^*(B) < \min(\lambda(A), \lambda(B))$,
%	\item $\lambda^*(A)=\lambda(A)<\min(\lambda(B), \lambda^*(B))$,
%	\item $\lambda^*(A)=\lambda^*(B) = \min(\lambda(A), \lambda(B))$,
%	\item $\lambda^*(A)=\lambda(A) = \lambda(B)<\lambda^*(B)$.
%\end{enumerate}
%}

As a consequence of Theorem~\ref{th:mainmod}, when $(X,\Mfrak)$, $(X,\Nfrak)$, $A$ and $B$ satisfy Assumption \ref{assump:1}, when there are at least two distinct critical parameters, only some mutual positions 
for the critical parameters are possible. Suppose that $\lambda(A)<\lambda(B)$.
Then the following situations, are the only ones possible:
\begin{enumerate}
	\item $\lambda^*(A)=\lambda^*(B) < \lambda(A)$,
	\item $\lambda^*(A)=\lambda(A)<\min(\lambda(B), \lambda^*(B))$,
	\item $\lambda^*(A)=\lambda^*(B) = \lambda(A)$,
	\item $\lambda^*(A)=\lambda(A) = \lambda(B)<\lambda^*(B)$.
\end{enumerate}
Similarly, one can obtain the possibilities in the cases where $\lambda^*(A)<\lambda^*(B)$ and/or the roles of $A$ and $B$ are swapped. If $\lambda(A)=\lambda(B)$ and $\lambda^*(A)=\lambda^*(B)$, then nothing more can be inferred:
it may be that $\lambda(A)=\lambda^*(A)$, $\lambda(A)<\lambda^*(A)$ or $\lambda(A)>\lambda^*(A)$.

See also Figure \ref{fig:possible-lambda} for the possible reciprocal positions of these critical parameters, in the case when 
$\lambda(A)<\lambda(B)$.

\begin{figure}[h!]
	\begin{tikzpicture}[scale=1]
		
		%\node[above] at (6,0.3) {If $\lambda(A)< \lambda(B)$};
		
		\draw[->] (0,0) -- (8,0);
		\node[below] at (0,-0.15) {0};
		\draw[-] (0,-.1) -- (0,.1);
		\node[right] at (9,0) {Case (1): POSSIBLE};

		\draw[-] (2,-.1) -- (2,.1);
		\draw[-] (4.5,-.1) -- (4.5,.1);
		\node[below] at (2,-0.1) {$\lambda^*(A)=\lambda^*(B)$};
		\node[below] at (4.5,-0.1) {$\lambda(A)$};

		%%%%%%%%%%%%%%%%%%%%%%%%%%%%

		\draw[->] (0,-1) -- (8,-1);
		\node[right] at (9,-1) {Case (2): POSSIBLE};
		\draw[-] (0,-1.1) -- (0,-.9);
		\node[below] at (0,-1.15) {0};

		\draw[-] (4.5,-1.1) -- (4.5,-0.9);
		\draw[-] (7,-1.1) -- (7,-.9);
		
		\node[below] at (7,-1.1) {$\lambda^*(B)$};
		\node[below] at (4.5,-1.1) {$\lambda(A)=\lambda^*(A)$};

		%%%%%%%%%%%%%%%%%%%%%%%%%%%%

		\draw[->] (0,-2) -- (8,-2);
		\node[right] at (9,-2) {Case (3): POSSIBLE};
		\draw[-] (0,-2.1) -- (0,-1.9);
		\node[below] at (0,-2.15) {0};

		\draw[-] (4.5,-2.1) -- (4.5,-1.9);
		%\draw[-] (7,-2.1) -- (7,-1.9);
		
		%\node[below] at (7,-1.1) {$\lambda^*(B)$};
		\node[below] at (4.5,-2.1) {$\lambda(A)=\lambda^*(A)=\lambda^*(B)$};

		%%%%%%%%%%%%%%%%%%%%%%%%%%%%
		
		\draw[->] (0,-3) -- (8,-3);
		\node[right] at (9,-3) {NOT POSSIBLE};
		\draw[-] (0,-3.1) -- (0,-2.9);
		\node[below] at (0,-3.15) {0};

		\draw[-] (4.5,-3.1) -- (4.5,-2.9);
		\draw[-] (7,-3.1) -- (7,-2.9);
		
		\node[below] at (7,-3.1) {$\lambda^*(B)$};
		\node[below] at (4.5,-3.1) {$\lambda(A)$};
		\draw[-] (2,-2.9) -- (2,-3.1);
		\node[below] at (2,-3.1) {$\lambda^*(A)$};
		
		%%%%%%%%%%%%%%%%%%%%%%%%%%%%

		\draw[->] (0,-4) -- (8,-4);
		\node[right] at (9,-4) {NOT POSSIBLE};
		\draw[-] (0,-4.1) -- (0,-3.9);
		\node[below] at (0,-4.15) {0};
		
		\draw[-] (2,-3.9) -- (2,-4.1);
		\draw[-] (3.25,-4.1) -- (3.25,-3.9);
		\draw[-] (4.5,-4.1) -- (4.5,-3.9);
		
		\node[below] at (2,-4.1) {$\lambda^*(A)$};
		\node[below] at (3.25,-4.1) {$\lambda^*(B)$};
		\node[below] at (4.5,-4.1) {$\lambda(A)$};
		
		%%%%%%%%%%%%%%%%%%%%%%%%%%%%
				
		\draw[->] (0,-5) -- (8,-5);
		\node[right] at (9,-5) {NOT POSSIBLE};
		\draw[-] (0,-5.1) -- (0,-4.9);
		\node[below] at (0,-5.15) {0};
		
		\draw[-] (7,-4.9) -- (7,-5.1);
		%\draw[-] (3.25,-4.1) -- (3.25,-3.9);
		\draw[-] (4.5,-5.1) -- (4.5,-4.9);
		
		\node[below] at (7,-5.1) {$\lambda^*(A)$};
		%\node[below] at (3.25,-4.1) {$\lambda^*(B)$};
		\node[below] at (4.5,-5.1) {$\lambda(A)$};
		
		%%%%%%%%%%%%%%%%%%%%%%%%%%%%
		
	\end{tikzpicture}
	\caption{$(X,\Mfrak)$, $(X;\Nfrak)$, $A$ and $B$ satisfy Assumption \ref{assump:1}.
		The corresponding critical parameters of $(X,\Mfrak)$ are $\lambda(A), \lambda (B)$, while the ones of
		$(X,\Nfrak)$ are $\lambda^*(A)$, $\lambda^* (B)$. 
		Assume that $\lambda(A)< \lambda (B)$.
		The first three pictures illustrate the possible 
		positions of $\lambda^*(A)$, $\lambda^* (B)$, the others show two situations which are not possible.
	}
	\label{fig:possible-lambda}
\end{figure}

As a direct consequence of Theorem~\ref{th:mainmod}, we get the following corollary. %Its proof is straightforward and we omit it.
Note that, under the hypothesis $\lambda(A)<\lambda(B)$, the first two sentences of Corollary~\ref{cor:new1} together show that $\lambda^*(A)=\lambda^*(B)$ if and only if $\lambda^*(B) \le \lambda(A)$ (where $\lambda$ and $\lambda^*$ are the same critical parameters of Theorem~\ref{th:mainmod}).
\begin{Corollary}\label{cor:new1}
	Let 
	$(X,\Mfrak)$ and $(X,\Nfrak)$ be two irreducible GMBRW with critical parameters denoted by $\lambda(\cdot)$ and $\lambda^*(\cdot)$ respectively. Consider $A, B \subseteq X$ such that Assumption~\ref{assump:1} holds and suppose that $\lambda(A)\le \lambda(B)$.
	Then the following hold.
	\begin{enumerate}
		\item If  $\max(\lambda^*(A),\lambda^*(B)) \le \lambda(A)$,
		then $\lambda^*(A)=\lambda^*(B)\le \lambda(A)$.
		\item If $\lambda(A)< \lambda(B)$ then $\lambda^*(A) \le \lambda^*(B)$; moreover $\lambda^*(B)>\lambda(A)$ implies $\lambda^*(B)>\lambda^*(A)=\lambda(A)$.
		\item If $\lambda(A) < \lambda(B)$, then it is not possible that 
		$\lambda^*(B)< \lambda^*(A)$, nor 
		that $\lambda(A)<  \lambda^*(A)$.
	\end{enumerate}
\end{Corollary}
\begin{proof}
	\leavevmode
	\begin{enumerate}
		\item 
		In this case we have either 
		$\min(\lambda^*(A),\lambda^*(B)) = \lambda(A)$ (thus $\lambda^*(A)=\lambda^*(B)=\lambda(A)$), 
		or else $\min(\lambda^*(A),\lambda^*(B))<\lambda(A)$ which, according to Proposition~\ref{pro:pureweak-nonstrong}, implies $\lambda^*(A)=  \lambda^*(B) <\lambda(A)$. 
		\item	
		Suppose that $\lambda(A)< \lambda(B)$; the inequality $\lambda^*(A) \le \lambda^*(B)$ comes from Theorem~\ref{th:mainmod}(1). 
		%	Therefore $\min(\lambda^*(A),\lambda^*(B))=\lambda^*(A)$ and $\max(\lambda^*(A),\lambda^*(B))=\lambda^*(B)$.	
		If, in addition	$\lambda^*(B) > \lambda(A)$ holds then the conclusion follows from Theorem~\ref{th:mainmod}(2).
		\item It comes from Theorem~\ref{th:mainmod}(1).
	\end{enumerate}
\end{proof}

The following corollary states that a critical parameter $\lambda(A)$ is either maximal, when we fix the GMBRW $\Mfrak$, in the family
$\{\lambda(C)\}_{C\neq\emptyset}$, or it is maximal, when we fix $A$, in the family 
$\{\lambda^\Nfrak(A)\}_{\Nfrak}$, where $\Nfrak$ satisfies Assumption~\ref{assump:1}, for some
finite, nonempty set $B\subseteq X$.
\begin{Corollary}\label{cor:new2}
	Let 
	$(X,\Mfrak)$ be an irreducible GMBRW and 
	let $A\subseteq X$ be nonempty subset.
	Then either $\lambda^\Mfrak(A)=\lambda^\Mfrak_s$ (equivalently $\lambda^\Mfrak(A) \ge \lambda^\Mfrak(B)$ for all nonempty $B \subset X$), or $\lambda^\Mfrak(A)\ge\lambda^\Nfrak(A)$ for all generic irreducible GMBRWs $(X,\Nfrak)$ such that
	Assumption~\ref{assump:1} holds.\\
	In particular, if $\lambda^\Mfrak(A) < \lambda^\Mfrak_s$, then for all $\Nfrak\in [\Mfrak]$, either
	$\lambda^\Nfrak(A)=\lambda^\Nfrak_s\le\lambda^\Mfrak(A)$, or 
	$\lambda^\Nfrak(A)=\lambda^\Mfrak(A) < \lambda^\Nfrak_s$.
\end{Corollary}
\begin{proof}
	We observe that
	the equivalence between $\lambda^\Mfrak(A) \ge \lambda^\Mfrak(B)$ for every nonempty $B \subseteq X$ and $\lambda^\Mfrak(A)=\lambda_s^\Mfrak$, can be proved as follows: clearly for every nonempty $B$  we have $\lambda^\Mfrak(B) \le \lambda_s^\Mfrak$; conversely, if $\lambda^\Mfrak(A) \ge \lambda^\Mfrak(B)$ for every nonempty $B \subseteq X$ then, by taking $B$ finite and nonempty, we have $\lambda^\Mfrak(A) \ge \lambda^\Mfrak_s$ (while $\lambda^\Mfrak(A) \le \lambda^\Mfrak_s$ since $A$ is nonempty).
	
	Since $\lambda^\Mfrak(A)\le\lambda^\Mfrak_s$ when $A$ is nonempty, thus if 
	$\lambda^\Mfrak(A)\neq\lambda^\Mfrak_s$, then $\lambda^\Mfrak(A) < \lambda^\Mfrak_s=\lambda^\Mfrak(B)$ where $B \subseteq X$ is finite and nonempty. Apply Theorem~\ref{th:mainmod}, with such a set $B$, to obtain
	that $\lambda^\Mfrak(A) < \lambda^\Mfrak_s$ implies  $\lambda^\Nfrak(A)\le\lambda^\Mfrak(A)$.\\
	Moreover, when $\Nfrak\in [\Mfrak]$, it means that $\Delta_{\Mfrak,\Nfrak}$ is finite, thus
	$\Delta_{\Mfrak,\Nfrak}\stackrel{\Mfrak}{\Rightarrow} B$ and $\Delta_{\Mfrak,\Nfrak}\stackrel{\Nfrak}{\Rightarrow} B$ for all nonempty sets $B$.
	Thus, again by Theorem~\ref{th:mainmod}, $\lambda^\Mfrak(A) < \lambda^\Mfrak_s$ implies the claim.
\end{proof}

The fact that $\lambda(A)<\lambda(B)$ leads to situations where processes survive in $A$ while going extinct in $B$, with positive probability, as the following proposition shows.
\begin{Proposition}
	\label{pro:nonstrong}
	Let 
	$(X,\Mfrak)$ be an irreducible GMBRW, with critical parameters denoted by $\lambda(\cdot)$.
	Let $A, B \subseteq X$ be such that $\lambda(A)<\lambda(B)$.
	Suppose that $(X,\Nfrak)$ is another irreducible GMBRW, with critical parameters denoted by $\lambda^*(\cdot)$, and that
	Assumption~\ref{assump:1} holds.
\begin{enumerate}
\item	
	If $\lambda^*(B) < \lambda(B)$,
	then for all $\lambda \in  (\max(\lambda(A),\lambda^*(B)), \, \lambda(B) )$ for the $(X,\Nfrak)$-BRW with this fixed $\lambda$, we have 
	$\pr^x(\mathcal{S}(B))>0$ and $\pr^x(\mathcal{S}(A)\cap\mathcal{E}(B)) >0$, for all $x \in X$.
\item	
	If $\lambda_w<\lambda_s$ and $\lambda^*_s < \lambda_s$, then for all $\lambda \in 
	(\max(\lambda_w, \lambda^*_s),\, \lambda_s  )$, the $(X,\Nfrak)$-BRW with this fixed $\lambda$
	has nonstrong local survival, starting from all $x \in X$.
\end{enumerate}	
\end{Proposition}
\begin{proof}
\begin{enumerate}
\item
	If $\lambda \in \big ( \max(\lambda(A),\lambda^*(B)), \lambda(B)\big )$ we have 
	$\mathbf{q}^{\Mfrak}(B|\lambda)=\mathbf{1}>\mathbf{q}^{\Mfrak}(A|\lambda)$. According to Theorem~\ref{th:modifiedBRW},
	$\mathbf{1}>\mathbf{q}^{\Nfrak}(B|\lambda)>\mathbf{q}^{\Nfrak}(A|\lambda)$, where the first inequality follows from $\lambda > \lambda^{{\Nfrak}}(B)$ 
	and implies that $\pr^x(\mathcal{S}(B))>0$, for all $x\in X$.
	By Theorem \ref{cor:equivalence}, $A\stackrel{\Nfrak}{\nRightarrow} B$
	and this means that $\pr^x(\mathcal{S}(A)\cap\mathcal{E}(B)) >0$, for all $x \in X$.
\item	
	The second part follows easily since it is enough to choose $A=X$ and a nonempty finite set $B$, then
	$\lambda(A)=\lambda_w$, $\lambda(B)=\lambda_s$,  $\lambda^*(A)=\lambda^*_w$ and $\lambda^*(B)=\lambda^*_s$.	\\
	\end{enumerate}	
\end{proof}

An application of these results is given in the following section.

\subsection{Global and local critical parameters for GMBRWs 
%and continuous-time BRWs 
under finite modifications}
\label{subsec:max}

Here we specify our general results for irreducible GMBRWs by considering the case of the usual critical parameters $\lambda_w$ and $\lambda_s$ and finite modifications. The following results can be stated, for instance, for the subclass of discrete-time counterpart of continuous-time BRWs as in \cite{cf:BZ2025} (see Remark~\ref{rem:continuoustime} for details).

The first result follows from Proposition~\ref{pro:pureweak-nonstrong} and Corollary~\ref{cor:new1} (see also \cite[Corollary 2]{cf:BZ2025}).

\begin{Corollary}
\label{cor:classiccritical}
Let $(X,\Mfrak)$ and $(X,\Nfrak)$  be two irreducible GMBRWs such that $\Mfrak \sim \Nfrak$
 and denote the critical parameters by $\lambda(\cdot)$ and $\lambda^*(\cdot)$ respectively. The following are equivalent:
\begin{enumerate}
\item $\lambda_w^*<\lambda_w$,
\item $\lambda_s^*<\lambda_w$,
\item $\lambda_w^*=\lambda_s^*<\lambda_w$.
\end{enumerate}
Finally, if $\lambda_s^* \le \lambda_w$ then $\lambda_w^*=\lambda_s^*\le \lambda_w$; if, in addition, $\lambda_w<\lambda_s$ then $\lambda_s^* > \lambda_w$ implies $\lambda_s^*>\lambda_w^*=\lambda_w$.
\end{Corollary}

The second result deals with the maximality of the global critical parameter $\lambda_w$. If a class contains at least one GMBRW with pure global survival phase, then each and every one of them has the same parameter $\lambda_w$ which attains the maximum value inside the class. Roughly speaking, either $\lambda_w$ attains the maximum value or $\lambda_w=\lambda_s$. it was originally proven for continuous-time BRWs in \cite[Proposition 2]{cf:BZ2025} but now it follows from Theorem~\ref{th:mainmod} (or Corollary~\ref{cor:new1}) by choosing a finite nonempty $A$ and $B:=X$.

\begin{Corollary}\label{cor:maximality}
Let $(X,\Mfrak)$ such that $\lambda_w^\Mfrak < \lambda_s^\Mfrak$. Then for all $\Nfrak \in [\Mfrak]$ we have $\lambda_w^{\Nfrak} \le \lambda_w^\Mfrak$. Moreover,  for all $\Nfrak \in [\Mfrak]$ such that $\lambda_w^{\Nfrak} < \lambda_s^{\Nfrak}$,  we have $\lambda_w^{\Nfrak} = \lambda_w^\Mfrak$.	
\end{Corollary}

% \begin{Corollary}\label{cor:maximality}
% Let $(X,K) \in \mathcal{X}$ such that $\lambda_w^K < \lambda_s^K$. Then for all $K^* \in [K]$ we have $\lambda_w^{K^*} \le \lambda_w^K$. Moreover,  for all $K^* \in [K]$ such that $\lambda_w^{K^*} < \lambda_s^{K^*}$,  we have $\lambda_w^{K^*} = \lambda_w^K$.	
% \end{Corollary}
% % \begin{proof}%[Proof of Proposition~\ref{pro:maximality}]
% % By contradiction, if $\lambda_w^{K^*} > \lambda_w^K$, then, 
% % according to Proposition~\ref{pro:pureweak-nonstrong-CT}, we have
% % $\lambda_s^K=\lambda_w^K$, which contradicts the hypothesis.
% % From the previous part, it easily follows that if $\lambda_w^K < \lambda_s^K$ and $\lambda_w^{K^*} < \lambda_s^{K^*}$, then $\lambda_w^{K^*} = \lambda_w^K$.
% % \end{proof}

\noindent See \cite{cf:BZ2025} for a full review on the global critical value in a class.

The second result of this section deals with the maximality of the local critical parameter. hereafter, by \emph{local survival phase} we mean \emph{survival in a finite nonempty set} (i.e.~survival in every finite nonempty set since the process is irreducible).
In this case, if a class contains at least one GMBRW with pure global survival phase and no nonstrong local survival phase, then each and every one of them has the same global and local critical parameters which attain the maximum value in the class. It follows from Proposition~\ref{pro:maximality2} since $\lambda(A)=\lambda_s$ for every finite nonempty $A \subseteq X$ and every BRW. This result can be stated, in particular, for the subclass of discrete-time counterparts of continuous-time BRWs as described in Remark~\ref{rem:continuoustime}.

\begin{Corollary}\label{cor:nonstrong}  
Let $(X, \Mfrak)$ and $(X,\Nfrak)$ be two irreducible GMBRWs such that $\Mfrak \sim \Nfrak$.
 Suppose that $\max(\lambda^{\Nfrak}_s, \lambda^\Mfrak_w)<\lambda \le \lambda^\Mfrak_s$, then $(X, \Nfrak)$ has a nonstrong survival phase in every nonempty, finite set.
\end{Corollary}

These results imply the following.

\begin{Proposition}
\label{pro:maximality2}
Let $(X,\Mfrak)$ be
%and $(X,\Nfrak)$ two 
an irreducible GMBRW.
% such that $k_{xy}=\Nfrak_{xy}$ for all $x \in X \setminus A$
%where $A$ is a finite set.
\begin{enumerate}
\item If $\lambda^{\Mfrak}_w<\lambda^{\Mfrak}_s$ then, for all $\Nfrak \in [\Mfrak]$, either $\lambda^{\Nfrak}_w=\lambda^{\Mfrak}_w$ and $\lambda^{\Nfrak}_s \ge \lambda^{\Mfrak}_s$ or $(X,\Nfrak)$ exhibits a nonstrong local survival phase.		
\item  If 
$\lambda^{\Mfrak}_w<\lambda^{\Mfrak}_s$ and
%then $\lambda^{K}_w \ge \lambda^{K^*}_w$. Moreover, if 
$(X,\Mfrak)$ has no nonstrong local survival phase, then for all $\Nfrak\in [\Mfrak]$ we have $\lambda^{\Mfrak}_s \ge \lambda^{\Nfrak}_s$. Moreover, for all $\Nfrak \in [\Mfrak]$ such that $(X,\Nfrak)$ has no nonstrong local survival phase, we have $\lambda_s^{\Nfrak}=\lambda_s^{\Mfrak}$ and $\lambda_w^{\Nfrak}=\lambda_w^{\Mfrak}$.
\end{enumerate}
\end{Proposition}

\begin{proof}
\leavevmode
\begin{enumerate}
\item The first part comes from Corollary~\ref{cor:nonstrong}. Indeed, since $\lambda^{\Mfrak}_w <\lambda^{\Mfrak}_s$, and $(X,\Nfrak)$ does not have nonstrong survival phase, then $\max(\lambda_s^{\Nfrak},  \lambda^{\Mfrak}_w) \ge \lambda^{\Mfrak}_s$, that is $\lambda_s^{\Nfrak} \ge \lambda^{\Mfrak}_s$. Moreover, from Corollary~\ref{cor:maximality}, $\lambda_w^{\Nfrak} \le \lambda^{\Mfrak}_w$, whence $\lambda_w^{\Nfrak} < \lambda_s^{\Nfrak}$; thus, again by Corollary~\ref{cor:maximality}, $\lambda_w^{\Nfrak} = \lambda^{\Mfrak}_w$.
\item %The first part is simply Proposition~\ref{pro:maximality}. 
We start by proving that $\lambda_s^\Mfrak \ge \lambda_s^{\Nfrak}$.
Since $\lambda^\Mfrak_w < \lambda^\Mfrak_s$ then, according to Corollary~\ref{cor:maximality}, either $\lambda^\Mfrak_w = \lambda^{\Nfrak}_w$ or $\lambda_s^{\Nfrak}=\lambda_w^{\Nfrak}<\lambda_w^\Mfrak$. In the second case 
the assertion is proved. In the first case, we can apply (1) by switching the role between $\Mfrak$ and $\Nfrak$: $\lambda_w^{\Nfrak}<\lambda_w^\Mfrak$, $(X,\Mfrak)$ does not exhibit a nonstrong survival phase therefore $\lambda^\Mfrak_s \ge \lambda^{\Nfrak}_s$.

We now prove the second assertion of (2). We already proved that $\lambda^\Mfrak_s \ge \lambda^{\Nfrak}_s$. Now, since $\lambda^\Mfrak_w < \lambda^\Mfrak_s$ and $(X,\Nfrak)$ does not have a nonstrong local survival phase, then from (1) we have $\lambda^{\Nfrak}_w=\lambda^{\Mfrak}_w$, $\lambda^{\Nfrak}_s \ge \lambda^{\Mfrak}_s$. Therefore, $\lambda^{\Nfrak}_w=\lambda^{\Mfrak}_w$ and $\lambda^{\Nfrak}_s = \lambda^{\Mfrak}_s$.

\end{enumerate}
\end{proof}

\begin{Remark}
\label{rem:continuoustime}
The results of Section~\ref{sec:survivalprob}
apply naturally to a subclass of GWBRWs, namely the discrete-time counterparts of continuous-time BRWs and of
ageing BRWs. 
In this sense, our results significantly extend those obtained in \cite{cf:BZ2025}.
\end{Remark}

As an application of our results to the subclass of continuous-time BRWs (as explained in the previous remark), we can prove the following.
We recall that $(X,K)$ is \textit{quasitransitive} if a finite $X_0\subset X$ exists such that for every $x\in X$, there is
a bijective map $\gamma_x:X\to X$ satisfying $\gamma_x^{-1}(x)
\in X_0$ and $k_{yz}=k_{\gamma_x y\,\gamma_x z}$ for all
$y,z$. For example, if the rates are translation-invariant, then $(X,K)$ is quasi-transitive (actually, it is transitive).

\begin{Corollary}\label{cor:pureglobal-nonostronglocal}
\leavevmode
\begin{enumerate}
\item 
An irreducible continuous-time BRW with pure global survival phase and an irreducible continuous-time BRW with no pure global survival phase and no nonstrong local survival phase cannot be in the same equivalence class with respect to $\sim$.
\item An irreducible continuous-time BRW with pure global survival phase and an irreducible quasitransitive continuous-time BRW with no pure global survival phase cannot be in the same equivalence class with respect to $\sim$.
\end{enumerate}
\end{Corollary}

\begin{proof}
\leavevmode
\begin{enumerate}
\item We prove the result by contradiction.
Let $(X,K)$ be such that $\lambda_w^K<\lambda_s^K$ and $(X,K^*)$ be such that $\lambda_w^{K^*}=\lambda_s^{K^*}$ and it does not have a nonstrong local survival phase. If $K{\sim}K^*$
 then, according to Proposition~\ref{pro:maximality2}(1) $\lambda_w^{K^*}=\lambda_w^K<\lambda_s^{K}\le\lambda_s^{K^*}$ and this contradicts $\lambda_w^{K^*}=\lambda_s^{K^*}$. Therefore, the two processes are not in the same equivalence class.

\item It follows from (1) if we prove that a quasitransitive BRW has no nonstrong local survival phase. This is a consequence of \cite[Corollary 3.2]{cf:BZ14-SLS}; indeed, for all $x \in X$ either $\mathbf{q}(x,x)=1$ or $\mathbf{q}(x,x)=\mathbf{q}(x,X)$. Since the BRW is irreducible, then either
$\mathbf{q}(x,x)=1$ for all $x \in X$ or $\mathbf{q}(x,x)=\mathbf{q}(x,X)$ for all $x \in X$. This yields the claim.    
\end{enumerate}
\end{proof}

\section{An example: the tree $\mathbb{T}_d \oplus \mathbb{T}_q$}
\label{sec:examples}

In this section, we provide an example showing that $\lambda_w <\lambda(A)<\lambda_s$ is possible for some $A\subseteq X$ and for irreducible GMBRWs on $X$ and we study how these critical values are affected by a specific local modification. The process, in this case, is the discrete-time counterpart of a continuous-time BRW. 

\begin{Example}
\label{ex:T5T6}
Let  $d,q \in \mathbb{N}$ be such that $d \ge 6$ and $ 2\sqrt{d-1}<q<d$.
%; clearly $q \ge 5$. 
Denote by
$X$ the vertex set of the graph $\mathbb{T}_d \oplus \mathbb{T}_q$ 
obtained by identifying the roots of $\mathbb{T}_d$ and $\mathbb{T}_q$, while the set of edges is the union of the standard sets of edges of the trees $\mathbb{T}_q$ and $\mathbb{T}_d$.  Let $K$ be the adjacency matrix of the graph: to each edge we associate a rate equal to 1. 
Then, $\lambda_w=1/d<1/q=\lambda(\mathbb{T}_q)<1/(2\sqrt{d-1})=\lambda_s$.
\end{Example}

\begin{proof}
The graph is obtained gluing the trees by the roots; we denote the new root by $o$.
The root has $d+q$ neighbors, while the number of neighbors of $x$ is $d$ (resp.~$q$) if $x$ is not the root and belongs to the subset $\mathbb{T}_d$ (resp.~$\mathbb{T}_q$).

It is very easy to prove that, since the removal of $o$ disconnects $\mathbb{T}_d$ and $\mathbb{T}_q$ then $\Phi^X_{\lambda}(o,o|1)=\Phi^{\mathbb{T}_d}_{ \lambda}(o,o|1)+\Phi^{\mathbb{T}_q}_{\lambda}(o,o|1)$ where $\Phi^{\mathbb{T}_n}_{\lambda}(o,o|1)=n \big (1-\sqrt{1-4(n-1)\lambda^2} \big )/(2(n-1))$ for all $n \in \mathbb{N}$, $n \ge 2$.

We now prove that $\lambda_w=1/d$, $\lambda_s=1/(2\sqrt{d-1})$, $\lambda(\mathbb{T}_q)=1/q$ and $\lambda(\mathbb{T}_d)=1/d$.
\begin{enumerate}  
\item $\lambda_s=1/(2\sqrt{d-1})$. Indeed, by equation~\eqref{eq:lambdas1}
\[
\lambda_s(x)=
\sup 
\Big \{ \lambda %\geq 0
\in {\mathbb R}: q \frac{{1-\sqrt{1-4(q-1)\lambda^2}}}{2(q-1)}+d \frac{1-\sqrt{1-4(d-1)\lambda^2}}{2(d-1)} < 1 \Big \}
\]
where $d \ge 6$. Clearly $\lambda_s \le 1/(2\sqrt{d-1})$ which is the radius of convergence of the power series $\Phi^{X}_{\lambda}(x,x|1)$ as a function of $\lambda$. On the other hand, 
\[
\Phi^X_{1/(2\sqrt{d-1})} \big (o,o|1 \big)=q \frac{{1-\sqrt{1-(q-1)/(d-1)}}}{2(q-1)}+\frac{d}{2(d-1)}<1
\]
where the last inequality holds since $d>q \ge 5$ (note that, for every fixed $q \in \mathbb{N}, \, q \ge 2$, the function $d \mapsto q \frac{{1-\sqrt{1-(q-1)/(d-1)}}}{2(q-1)}+\frac{d}{2(d-1)}$ is strictly decreasing in $[q, +\infty)$). More precisely, the R.H.S.~of the previous equation is bounded from above (for all $q,d$ such that $q<d$) by the same expression where $d=q+1$, that is  $\Phi^X_{ 1/(2\sqrt{d-1})} \big (o,o|1 \big) \le q \frac{{1-\sqrt{1/q}}}{2(q-1)}+\frac{q+1}{2q}=1-\frac{1}{2(\sqrt{q}+1)}+\frac{1}{q}$. It is easy to show that $1-\frac{1}{2(\sqrt{q}+1)}+\frac{1}{q}<1$ for all $q \ge 5$.

\item $\lambda_w=1/d$. 
Take  $\lambda<1/d$: since $\lambda < \lambda_s=1/(2\sqrt{d-1})$, there is local extinction. Thus, if there were global
survival, it would be pure global survival, that is, $\mathbf{q}(o,X|\lambda)<\mathbf{q}(o,o|\lambda)=\mathbf 1$, and $o\stackrel{\boldmu_\lambda}{\nRightarrow} X$.
By Theorem \ref{cor:equivalence}, this would imply that there is a positive probability of surviving starting from some $x \in X$, $x\neq o$, without visiting $o$ and that the process restricted to  $\mathbb{T}_d$ (if $x \in \mathbb{T}_d$) or to $\mathbb{T}_q$ ($x \in \mathbb{T}_q$) survives, but this is false since $\lambda<\lambda_w(\mathbb{T}_d)=1/d<1/q=
\lambda_w(\mathbb{T}_q)$. Thus, $\lambda_w \ge 1/d$.
Conversely, if $\lambda >1/d$, then there is survival even if the process is restricted to $\mathbb{T}_d$.  Thus $\lambda_w  \le 1/d$; hence $\lambda_w = 1/d$.

\item $\lambda(\mathbb{T}_q)=1/q$. Clearly $\lambda(\mathbb{T}_q) \le 1/q$ since if $\lambda >1/q$ the process restricted to $\mathbb{T}_q$ survives. Take  $\lambda<1/q$. If there were survival, it would be pure global survival, since $\lambda < \lambda_s=1/(2\sqrt{d-1})$. Therefore, 
by Theorem \ref{cor:equivalence} there would be
a positive probability of surviving in $\mathbb{T}_q$, starting from some $x \in X$, without visiting $o$ (clearly $x \in \mathbb{T}_q$). This would imply that the process restricted to $\mathbb{T}_q$ survives, but this is false. 

\item Similarly, we could prove that $\lambda(\mathbb{T}_d)=1/d$.
\end{enumerate}

In the end, we have $\lambda_w=1/d<1/q=\lambda(\mathbb{T}_q)<1/(2\sqrt{d-1})=\lambda_s$.
\end{proof}

It is easy to see that $q=5$ and $d=6$ satisfy the requirements of Exmaple \ref{ex:T5T6}. 
We can add a loop at the root and study how this addition affects the critical parameters: see the following example.

\begin{Example}\label{ex:T5T6loop}
Consider  the graph $\mathbb{T}_d \oplus \mathbb{T}_q$ of the previous example.
Let  $K^*$ coincide with the adjacency matrix $K$ except for the entry $k^*_{oo}:=k>0$. Denote by $\lambda_w^*$, $\lambda^*(\mathbb{T}_q)$, and $\lambda_s^*$ the critical values of the BRW $(X,K^*)$, viewed as functions of $k>0$. 
There exist constants $0<k_1<k_2<k_3$ such that
\begin{itemize}
\item for all $k \in [0,k_1]$ we have $\lambda_w^*=\lambda_w<\lambda^*(\mathbb{T}_{q})=\lambda(\mathbb{T}_{q})<\lambda_s^*=\lambda_s$;
\item for all $k \in (k_1, k_2)$ we have $\lambda_w^*=\lambda_w<\lambda^*(\mathbb{T}_{q})=\lambda(\mathbb{T}_{q})<\lambda_s^*<\lambda_s$;
\item if $k=k_2$ we have $\lambda_w^*=\lambda_w<\lambda^*(\mathbb{T}_{q})=\lambda(\mathbb{T}_{q})=\lambda_s^*<\lambda_s$;
\item for all $k \in (k_2, k_3)$ we have $\lambda_w^*=\lambda_w<\lambda^*(\mathbb{T}_{q})=\lambda_s^*<\lambda(\mathbb{T}_{q})<\lambda_s$;
\item if $k=k_3$ we have $\lambda_w^*=\lambda_w=\lambda^*(\mathbb{T}_{q})=\lambda_s^*<\lambda(\mathbb{T}_{q})<\lambda_s$;
\item for all $k \in (k_3, +\infty)$ we have $\lambda_w^*=\lambda^*(\mathbb{T}_{q})=\lambda_s^*<\lambda_w<\lambda(\mathbb{T}_{q})<\lambda_s$.
\end{itemize}
\end{Example}
\begin{proof}
 The statement follows using Corollaries~\ref{cor:classiccritical}~and~\ref{cor:maximality} and arguments analogous to those in \cite[Example~1]{cf:BZ2025}.
\end{proof}

Let us note that it is possible to compute $k_1, k_2,$ and $k_3$ explicitly, but the required calculations
are lengthy and offer little additional insight.

\begin{Remark}
\label{rem:sweep}
In the above example, we studied a family of GRBRWs parametrized by $k>0$. This leads to the following general observation based on Corollary~\ref{cor:new2}. Let $\{\Mfrak_t\}_{t \in I}$ be a family of GMBRWs indexed by $t \in I \subseteq \mathbb{R}$ such that $[\Mfrak_t]=[\Mfrak_r]$ for all $t,r \in I$. Suppose that the map $t \mapsto \lambda^{\Mfrak_t}_s$ is nonincreasing and continuous. Informally, $\lambda^{\Mfrak_t}_s$ behaves like a sweeping threshold: for all sets $A$ with $\lambda^{\Mfrak_t}(A)<\lambda^{\Mfrak_t}_s$ (for some $t\in I$), the map $t\mapsto \lambda^{\Mfrak_t}(A)$ remains constant up to the smallest value $r\in I$ such that $\lambda^{\Mfrak_r}(A)=\lambda^{\Mfrak_r}_s$. For all $t\ge r$ the two critical values coincide, that is, $\lambda^{\Mfrak_t}(A)=\lambda^{\Mfrak_t}_s$.
\end{Remark}

The next example is a variation of the previous one. The aim is to modify the reproduction laws so that the resulting BRWs are germ ordered but not pgf ordered 
%with respect to the pgf order
 (and therefore not stochastically ordered). 
To this end, we begin with a general construction. Let $(X,P)$ be a random walk represented by the stochastic matrix $P=\big (p(x,y) \big )_{x,y \in X}$ and assume that 
$
\mathrm{deg}(x):=\#\{y \in X \colon p(x,y)>0\}<+\infty \quad \text{for all } x \in X$,
where $\mathrm{deg}(x)$ denotes the \emph{degree} of $x$, that is, the number of sits where a walker can jump from $x$.
%its neighbors. 
For $\lambda>0$ and $n \in \mathbb{N}\setminus\{0\}$, consider the generating function
\[
g_{n,\lambda}(t):=\frac{\lambda+\exp\big(n\alpha_\lambda (t-1)\big)}{1+\lambda},
\]
where the function $\lambda \mapsto \alpha_\lambda$ will be specified later so that the map $\lambda \mapsto \alpha_\lambda/(1+\lambda)$ is strictly increasing. This is the generating function of a random variable defined as follows: with probability $\lambda/(1+\lambda)$ the outcome is $0$, while with probability $1/(1+\lambda)$ it is the outcome of a Poisson random variable with parameter $n \alpha_\lambda$.

Let $(X,\boldmu_\lambda)$ be the BRW defined as follows: given a particle at site $x$, the number of offspring is distributed according to the law with generating function $g_{\mathrm{deg}(x),\lambda}$, and the offspring are independently placed according to the transition probabilities $p(x,\cdot)$. 
It is a GMBRW, but it is not stochastically ordered. 

\begin{Example}\label{ex:nonstochmon}
Let  $X=\mathbb{T}_d \oplus \mathbb{T}_q$, $P$ be the transition matrix of the simple random walk and
\[
g_{n,\lambda}(t)=\frac{\lambda+\exp\big(2n\lambda(t-1)\big)}{1+\lambda}.
\]
The associated family $(X,\boldmu_\lambda)$ is a GMBRW, meaning that if $\lambda\ge\lambda^*$, then
 $\boldmu_\lambda \gegerm \boldmu_{\lambda^*}$, while $\boldmu_\lambda \not\!\!\!\gepgf \boldmu_{\lambda^*}$.
 \end{Example}
\begin{proof}
Observe that, since $g_{n,\lambda}'(1)=2 n \lambda/(1+\lambda)$
% $g_{n,\lambda}'(1)=n \alpha_\lambda/(1+\lambda)$, 
then $\lambda \mapsto g'_{n,\lambda}(1)$ is strictly increasing. Hence, by \cite[Proposition~2.3]{cf:BZgerm}, we have $\boldmu_\lambda \gegerm \boldmu_{\lambda'}$, for all $\lambda > \lambda'>0$.
Indeed, the entries of the first-moment matrix of $(X, \boldmu_\lambda)$ are 
$m_{xy}(\lambda)=2\lambda/(1+\lambda)$,
%$m_{xy}(\lambda)=\mathrm{deg}(x) p(x,y)\alpha_\lambda/(1+\lambda)$. Henceforth, let $(X,P)$ be a simple random %walk on a locally finite graph $(X,E_X)$, namely, $p(x,y)=1/\mathrm{deg}(x)$ if $y$ is a neighbor of $x$ and $0$ %otherwise. In this case  $m_{xy}(\lambda)=\alpha_\lambda/(1+\lambda)$
 if $y$ is a neighbor of $x$ and $0$ otherwise. 
%Let us fix $X=\mathbb{T}_d \oplus \mathbb{T}_q$; s
Straightforward computations yield
\[
\Phi_\lambda(x,x| 1)
= q \frac{1-\sqrt{1-4(q-1)\big(2\lambda/(1+\lambda)\big)^2}}{2(q-1)}
+ d \frac{1-\sqrt{1-4(d-1)\big(2\lambda/(1+\lambda)\big)^2}}{2(d-1)}.
\]
%\[
%\Phi_\lambda(x,x| 1)
%= q \frac{1-\sqrt{1-4(q-1)\big(\alpha_\lambda/(1+\lambda)\big)^2}}{2(q-1)}
%+ d \frac{1-\sqrt{1-4(d-1)\big(\alpha_\lambda/(1+\lambda)\big)^2}}{2(d-1)}.
%\]
%Indeed, the expected number of offspring along each edge is now $\alpha_\lambda/(1+\lambda)$, whereas in the previous example it was equal to $\lambda$. 
%In order to obtain explicit expressions, let us set $\alpha_\lambda:=2\lambda$. With this choice,
%\[
%g_{n,\lambda}(t)=\frac{\lambda+\exp\big(2n\lambda(t-1)\big)}{1+\lambda}
%\]
and the map $\lambda \mapsto 2\lambda/(1+\lambda)$ is strictly increasing. Hence, the family $\Mfrak:=\{(X,\boldmu_\lambda)\}_{\lambda>0}$ is  monotone with respect to the germ order. 
On the other hand, since
$\lambda \mapsto g_{n,\lambda}(0)=(\lambda +\exp(-2n\lambda))/(1+\lambda)$ is eventually strictly increasing as $\lambda \to +\infty$, the family is not monotone with respect to the pgf order.

In this case, the relevant computations can be derived from those of the previous example by substituting $\lambda$
 with $2\lambda/(1+\lambda)$. Thus the new critical values are obtained by applying the function $f(x):=x/(2-x)$ to the corresponding previous ones. Consequently,
$\lambda_w={1}/({2d-1})$, $\lambda_s={1}/({4\sqrt{d-1}-1})$, $\lambda(\mathbb{T}_q)={1}/({2q-1})$ and $\lambda(\mathbb{T}_d)={1}/({2d-1})$.
\end{proof}
Further examples can be constructed; however, they fall outside the scope of the present paper.

\end{document}